%% file: Mass_Center_and_Pappus.tex
\documentclass{amsart}
\usepackage{float}
\usepackage{subcaption}

\usepackage{xcolor,mathtools,amssymb,amsthm,thmtools}
\usepackage{amsmath,amssymb,mathtools}
\usepackage{kotex,bm}
\usepackage{hyperref}

\usepackage{tikz,tikz-3dplot}
\usetikzlibrary{arrows.meta,intersections}
\tikzset{
  vechead/.style={-{Stealth[scale=0.8]}}
}

\newtheorem{theorem}{Theorem}
\newtheorem{lemma}{Lemma}

\newtheorem{prop}{Proposition}

\theoremstyle{definition}
\newtheorem*{definition}{Definition}

\newcommand{\vcz}{\bm{\mathrm{0}}}
\newcommand{\vc}[1]{\bm{\mathrm{#1}}}
\newcommand{\dd}{\mathrm{d}}

\DeclareMathAlphabet{\mathdutchcal}{U}{dutchcal}{m}{n}
\SetMathAlphabet{\mathdutchcal}{bold}{U}{dutchcal}{b}{n}
\DeclareMathAlphabet{\mathdutchbcal}{U}{dutchcal}{b}{n}

\begin{document}

\title{Uniqueness of non-Euclidean Mass Center System and Generalized Pappus' Centroid Theorems in Three Geometries}
\renewcommand{\shorttitle}{Mass center system and Pappus theorem}

\author{Yunhi Cho and Hyounggyu Choi}

\begin{abstract}
G.A. Galperin introduced the axiomatic mass center system for finite point sets in spherical and hyperbolic spaces, proving the uniqueness of the mass center system. In this paper, we revisit this system and provide a significantly simpler proof of its uniqueness. Furthermore, we extend the axiomatic mass center system to manifolds. As an application of our system, we derive a highly generalized version of Pappus' centroid theorem for volumes in three geometries—Euclidean, spherical, and hyperbolic—across all dimensions, offering unified and notably simple proofs for all three geometries.
\end{abstract}
\maketitle

\section{Introduction}
The Greek mathematician Archimedes established the concept of the \emph{axiomatic mass center system} in Euclidean space (see pages 502-509 of \cite{archimedes}). With the advent of non-Euclidean geometry, the concept of the mass center naturally extended to non-Euclidean spaces. However, surprisingly, the mass center in non-Euclidean spaces has not been extensively studied, with only a few results available. For instance, in \cite{bjerre} and \cite{fog}, the spherical mass center was considered, while in \cite{andrade} and \cite{bonola}, the hyperbolic lever law, which is closely related to the hyperbolic mass center, was explored. After several related works, papers such as \cite{choe} and \cite{gray} have emerged, subtly relying on the concept of the non-Euclidean mass center. However, these papers often fail to recognize the correlation between the concepts they employ and the non-Euclidean mass center, instead treating it merely as an ancillary tool for their theoretical developments.

G. A. Galperin has set up the concept of axiomatic mass center system on the spherical space $\vc S^n$ and on the hyperbolic space $\vc H^n$ in his article \cite{galperin}. He started from a system of axioms that a system so called a \emph{mass center system} plausibly satisfies. Among the axioms, for example, one is that the mass center system must be isometry invariant. After introducing a system of axioms, he showed that there exists a unique mass center system satisfying the given system of axioms by using the D'Alembert's functional equation defined on $\mathbb{R}$:  $ f(x+y)+f(x-y)=2 f(x)f(y)
$.

Archimedes also proposed the axioms for mass center. There is a notable thing that must be mentioned. He basically assumed that the mass of the mass center is same as the total mass of the given points. However, Galperin excluded this assumption from the system of axioms. This is important. In result we will see, it is shown that the mass of the mass center is different from the total mass for spherical or hyperbolic cases.

The system of mass center is very simple as it will be presented. In Euclidean space, it is hopefully natural that the  mass of the mass center is same as the total mass of the concerned set. This is true in result. However, it will be turned out that the mass of the mass center is not necessarily same as the total mass in the cases for non-Euclidean spaces. At first glance, it seems surreal but after a while it becomes real. From now on, in this reason, we call the mass of the mass center the \emph{centered mass}.    

We will present an alternative proof for the uniqueness of the mass center for finite system. It seems that our proof is considerably simpler than the proof of Galperin. However, it is wholly a follower's benefit. Galperin's consideration on the concept of axiomatic mass center system is great and absolute. Additionally, we propose the axiomatic mass center system for manifolds and prove the uniqueness. By the way, Choe and Gulliver \cite{choe} reached to the concept of \emph{modified volume} in their own study on the isoperimetric inequality. The modified volume is nothing other than the centered mass. We are not sure that they considered the concept of centered mass.
As far as we know, Galperin is the first one who showed the uniqueness of the system and explicitly presented the concept of centered mass.

More than thirty years has passed after professor Galperin's work. There are not many applications of this work. We add one. That is \emph{Pappus' centroid theorem on non-Euclidean spaces}. In fact, already there is a very significant work of Alfred Grey and Vincent Miquel \cite{gray} on  Pappus' centroid theorem on non-Euclidean spaces. Their result contains much portion of ours but not all of ours. Citing Choe and Gulliver \cite{choe}, they used the concept of modified volume. It is also unsure that they considered the concept of centered mass.

Deriving volume formulas for geometric objects in non-Euclidean spaces is considerably more challenging compared to Euclidean space. Explicit volume formulas for non-Euclidean solids are generally unavailable, except in certain special cases, such as ideal polyhedra. In this study, we utilize the powerful tool of Pappus’ theorem to derive explicit volume formulas for four previously unknown types of non-Euclidean solids, including the right circular cone and the torus in three-dimensional space.

\section{Preliminary and Terminology}\label{section:models of spaces}
In this article, the space $\mathbb{X}^n$ denotes the $n$-dimensional Euclidean space $\mathbb{E}^n$, the spherical space $\mathbb{S}^n$, or the hyperbolic space $\mathbb{H}^n$. These spaces will be embedded in the ambient space $\mathbb{R}^{n+1}$. In particular, we adopt the \emph{hyperboloid model} for the hyperbolic space. In this section, we briefly outline the necessary background, most of which is well known.

\subsection{Spaces and bracket notation}
Let the \((n+1) \times (n+1)\) matrix \(J_{n+1}\) be
\[
J_{n+1} = \left(
\begin{array}{c|c}
  I_n & \vcz \\
  \hline
  \vcz & \delta_X
\end{array}
\right)_{(n+1) \times (n+1)},
\quad
\delta_X = \begin{cases}
    1, & \text{if } \mathbb{X}^n = \mathbb{E}^n \text{ or } \mathbb{S}^n, \\
    -1, & \text{if } \mathbb{X}^n = \mathbb{H}^n.
\end{cases}
\]
Let the symmetric bilinear form \(\langle \cdot, \cdot \rangle\) be defined on \(\mathbb{R}^{n+1} \times \mathbb{R}^{n+1}\) by
\[
\langle \vc{a}, \vc{b} \rangle = \vc{b}^T J_{n+1} \vc{a}.
\]

For integers \(n \geq 1\), let the Euclidean space \(\mathbb{E}^n\), the spherical space \(\mathbb{S}^n\), and the hyperbolic space \(\mathbb{H}^n\) lie in the ambient space \(\mathbb{R}^{n+1}\), such that, for \(\vc{r} = (x_1, \dots, x_{n+1})^T\),
\begin{align*}
\mathbb{E}^n &= \{\vc{r} \in \mathbb{R}^{n+1} \mid x_{n+1} = 1\}, \\
\mathbb{S}^n &= \{\vc{r} \in \mathbb{R}^{n+1} \mid \langle \vc{r}, \vc{r} \rangle = 1\}, \\
\mathbb{H}^n &= \{\vc{r} \in \mathbb{R}^{n+1} \mid \langle \vc{r}, \vc{r} \rangle = -1, x_{n+1} > 0\}.
\end{align*}

A basis \(\{\vc{v}_1, \dots, \vc{v}_{n+1}\}\) for \(\mathbb{R}^{n+1}\) is called an \emph{orthonormal} basis if
\begin{equation}
  \begin{pmatrix} \vc{v}_1 & \cdots & \vc{v}_{n+1} \end{pmatrix}^T J_{n+1} \begin{pmatrix} \vc{v}_1 & \cdots & \vc{v}_{n+1} \end{pmatrix} = J_{n+1}.
\end{equation}

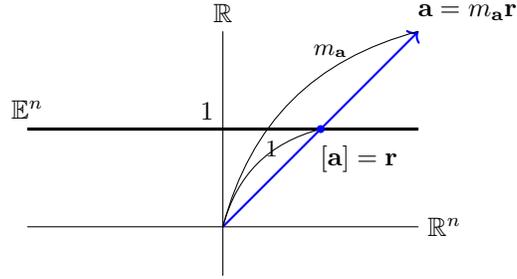
\begin{figure}[h!]
    \begin{center}
        \input{spaces_basic1.tex}
        \caption{Euclidean Space $\mathbb{E}^n$ lying on the ambient space $\mathbb{R}^{n+1}$, a material point $[\vc a]$ with mass $m_{\vc a}$, and the corresponding material vector $\vc a$. The concepts of material point, mass, and material vector will be introduced in the following section. }\label{efig1}
    \end{center}
\end{figure}

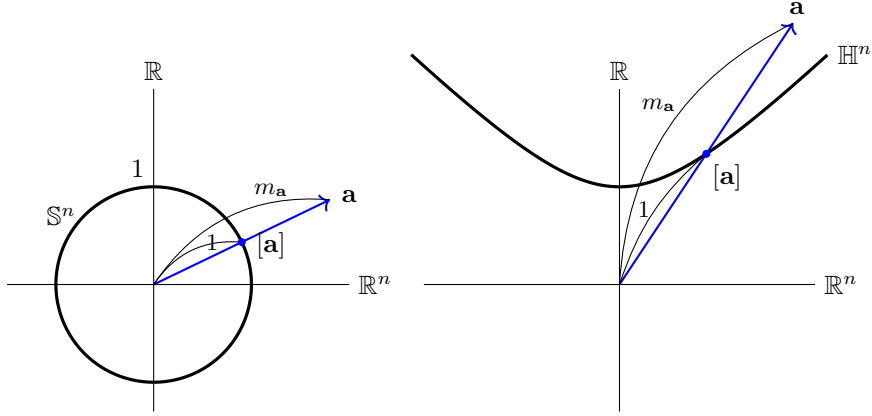
\begin{figure}[h!]
    \begin{center}
        \input{spaces_basic2.tex}
        \caption{Spherical Space $\mathbb{S}^n$ and hyperbolic space $\mathbb{H}^n$ lying on the ambient space $\mathbb{R}^{n+1}$, material points $[\vc a]$ with mass $m_{\vc a}$, and the corresponding material vectors $\vc a$}\label{efig2}
    \end{center}
\end{figure}

A vector \(\vc{v} \in \mathbb{R}^{n+1}\) is called a \emph{negative} vector or a \emph{positive} vector if, respectively, the value \(\langle \vc{v}, \vc{v} \rangle\) is negative or positive. By Sylvester's theorem of inertia, if \(\mathbb{X}^n = \mathbb{H}^n\), or equivalently \(\delta_X = -1\), then there is only one negative vector among any orthonormal basis for \(\mathbb{R}^{n+1}\).

See Figures \ref{efig1} and \ref{efig2}. Let \(\vc{r}\) be a point in \(\mathbb{X}^n\). Denote the \emph{cone} over \(\mathbb{X}^n\) as
\[
\mathbb{R}^+ \mathbb{X}^n := \{ m \vc{r} \in \mathbb{R}^{n+1} \mid m \geq 0, \, \vc{r} \in \mathbb{X}^n \}.
\]
Let \(\vc{a}\in \mathbb{R}^+ \mathbb{X}^n\), \(\vc{a} \neq \vcz\). Then there exists a unique \(m = m_{\vc{a}} > 0\) and a unique point \(\vc{r}\) in \(\mathbb{X}^n\) such that \(\vc{a} = m \vc{r}\). In this situation, we denote
\[
\vc{r} = [\vc{a}].
\]
In other words, we can write
\begin{equation}
\vc{a} = m_{\vc{a}} [\vc{a}].
\end{equation}
With this notation, we can express
\begin{equation}\label{f3}
\langle \vc{a}, \vc{a} \rangle = \delta_X {m_{\vc{a}}}^2, \quad \text{if } \mathbb{X}^n = \mathbb{S}^n \text{ or } \mathbb{H}^n.
\end{equation}

\subsection{Metric, distance, and parametrization}
The infinitesimal \emph{metric} $\dd{s}$ on $\mathbb{X}^n$ is defined by
\[
{\dd{s}}^2 = {\dd{x_1}}^2 + \dots + \delta_{X}{\dd{x_{n+1}}}^2, \quad (x_1, \dots, x_{n+1}) \in \mathbb{X}^n \subset \mathbb{R}^{n+1}.
\]
Let $\vc{a}, \vc{b} \in \mathbb{R}^{+}\mathbb{X}^n$. Then, the \emph{distance} $d\left([\vc{a}], [\vc{b}]\right)$ between $[\vc{a}]$ and $[\vc{b}]$ is given by
\begin{align}
  d\left([\vc{a}], [\vc{b}]\right) &= \sqrt{\left([\vc{a}] - [\vc{b}]\right) \cdot \left([\vc{a}] - [\vc{b}]\right)}, \quad & \text{if } \mathbb{X}^n = \mathbb{E}^n, \nonumber \\
  \cos\left\{ d\left([\vc{a}], [\vc{b}]\right) \right\} &= \frac{\langle \vc{a}, \vc{b} \rangle}{\sqrt{\langle \vc{a}, \vc{a} \rangle \langle \vc{b}, \vc{b} \rangle}}, \quad & \text{if } \mathbb{X}^n = \mathbb{S}^n, \label{distance1} \\
  \cosh\left\{ d\left([\vc{a}], [\vc{b}]\right) \right\} &= \frac{-\langle \vc{a}, \vc{b} \rangle}{\sqrt{\langle \vc{a}, \vc{a} \rangle \langle \vc{b}, \vc{b} \rangle}}, \quad & \text{if } \mathbb{X}^n = \mathbb{H}^n. \label{distance2}
\end{align}

For notational convenience, we introduce the functions $\sin_X$ and $\cos_X$ as follows:
\begin{equation}
 \sin_X(x) = \begin{cases}
    x, & \quad \text{if } \mathbb{X}^n = \mathbb{E}^n, \\
    \sin x, & \quad \text{if } \mathbb{X}^n = \mathbb{S}^n, \\
    \sinh x, & \quad \text{if } \mathbb{X}^n = \mathbb{H}^n.
\end{cases}   
\end{equation}
\begin{equation}
 \cos_X(x) = \begin{cases}
    1, & \quad \text{if } \mathbb{X}^n = \mathbb{E}^n, \\
    \cos x, & \quad \text{if } \mathbb{X}^n = \mathbb{S}^n, \\
    \cosh x, & \quad \text{if } \mathbb{X}^n = \mathbb{H}^n.
\end{cases}   
\end{equation}
With this convention, we can write
\begin{equation}
 \cos_X d\left([\vc{a}], [\vc{b}]\right) = \frac{\delta_X \langle \vc{a}, \vc{b} \rangle}{\sqrt{\langle \vc{a}, \vc{a} \rangle \langle \vc{b}, \vc{b} \rangle}}, \quad \text{if } \mathbb{X}^n = \mathbb{S}^n \text{ or } \mathbb{H}^n,
\end{equation}
or
\begin{equation}\label{f9}
 m_{\vc{a}} m_{\vc{b}} \cos_X d\left([\vc{a}], [\vc{b}]\right) = \delta_X \langle \vc{a}, \vc{b} \rangle, \quad \text{if } \mathbb{X}^n = \mathbb{S}^n \text{ or } \mathbb{H}^n,
\end{equation}

We now introduce the well-known parametrization for each of the spaces. By an abuse of notation, we denote the parametrization for each space of $\mathbb{S}^n$ and $\mathbb{H}^n$ using a single notation:
\[
(x_1, \dots, x_{n+1})^T = \Phi^n(\varphi_1, \dots, \varphi_n).
\]
For $\mathbb{X}^n = \mathbb{S}^n$ or $\mathbb{H}^n$, respectively,
\begin{equation*}
  \begin{cases}
  x_1 = \sin \varphi_1 \cdots \sin \varphi_{n-1} \sin \varphi_n, \\
  x_2 = \sin \varphi_1 \cdots \sin \varphi_{n-1} \cos \varphi_n, \\
  \vdots \\
  x_{n-1} = \sin \varphi_1 \sin \varphi_2 \cos \varphi_3, \\
  x_n = \sin \varphi_1 \cos \varphi_2, \\
  x_{n+1} = \cos \varphi_1.
  \end{cases}
  \quad
  \begin{cases}
  x_1 = \sinh \varphi_1 \sin \varphi_2 \cdots \sin \varphi_{n-1} \sin \varphi_n, \\
  x_2 = \sinh \varphi_1 \sin \varphi_2 \cdots \sin \varphi_{n-1} \cos \varphi_n, \\
  \vdots \\
  x_{n-1} = \sinh \varphi_1 \sin \varphi_2 \cos \varphi_3, \\
  x_n = \sinh \varphi_1 \cos \varphi_2, \\
  x_{n+1} = \cosh \varphi_1.
  \end{cases}
\end{equation*}
For $\mathbb{X}^n = \mathbb{E}^n$,
\begin{equation*}
  \Phi^n(x_1, \dots, x_n) = (x_1, \dots, x_n, x_{n+1})^T = (x_1, \dots, x_n, 1)^T.
\end{equation*}
The volume element $\dd V_n$ is given by
\[
\dd V_n =
\begin{cases}
  \dd{x_1} \cdots \dd{x_n}, & \text{if } \mathbb{X}^n = \mathbb{E}^n, \\
  \sin^{n-1} \varphi_1 \sin^{n-2} \varphi_2 \cdots \sin \varphi_{n-1} \dd{\varphi_1} \cdots \dd{\varphi_n}, & \text{if } \mathbb{X}^n = \mathbb{S}^n, \\
  \sinh^{n-1} \varphi_1 \sin^{n-2} \varphi_2 \cdots \sin \varphi_{n-1} \dd{\varphi_1} \cdots \dd{\varphi_n}, & \text{if } \mathbb{X}^n = \mathbb{H}^n.
\end{cases}
\]

Later, we will mention these parametrizations, but we don't need specific expressions. These parametrizations are regular except when $\mathbb{X}^n = \mathbb{S}^n$. In the spherical case, the parametrization has a non-empty set of singular points. This does not affect the calculation of volumes of submanifolds.

Let $M^k$ be a $k$-dimensional manifold in $\mathbb{X}^n$, and let $\vc{r}(s_0, s_1, \dots, s_{k-1})$ be an almost everywhere regular parametrization of $M^k$. Then, the \emph{volume element} $\dd{V}_k$ is given by
\[
\dd{V}_k = \sqrt{\det\left\{ (D\vc{r})^T J_{n+1} D\vc{r} \right\}} \, \dd{s_0} \dd{s_1} \cdots \dd{s_{k-1}},
\]
where $D\vc{r}$ denotes the Jacobian matrix
\[
\frac{\partial \vc{r}}{\partial (s_0, \dots, s_{k-1})} = \begin{pmatrix} \vc{r}_{s_0} & \vc{r}_{s_1} & \cdots & \vc{r}_{s_{k-1}} \end{pmatrix}.
\]

\subsection{Totally geodesic submanifold, hyperplane in spherical and hyperbolic space and  parametrization}

It is well known that the intersection of a \((k+1)\)-dimensional subspace of the ambient space \(\mathbb{R}^{n+1}\) with the space \(\mathbb{X}^n\) results in a \(k\)-dimensional \emph{totally geodesic submanifold} of \(\mathbb{X}^n\). A \emph{geodesic} is a 1-dimensional totally geodesic submanifold, and it is the intersection of a 2-dimensional subspace of the ambient space \(\mathbb{R}^{n+1}\) with the space \(\mathbb{X}^n\). A \emph{hyperplane} in \(\mathbb{X}^n\) is a totally geodesic submanifold of codimension 1. A hyperplane is the intersection of an \(n\)-dimensional subspace of the ambient space \(\mathbb{R}^{n+1}\) with the space \(\mathbb{X}^n\). We define the unit vector orthogonal to this subspace as the \emph{polar vector} of the hyperplane. The polar vector exhibits a sign ambiguity. A polar vector is a 
positive vector, and every positive vector corresponds to a unique hyperplane.

From now until the end of this subsection, we will assume that the space \(\mathbb{X}^n\) refers to either \(\mathbb{S}^n\) or \(\mathbb{H}^n\). 

Let \(\vc{c}\) be a point in \(\mathbb{X}^n\), and let \(\vc{p}\) be a vector orthogonal to \(\vc{c}\). Then, \(\vc{p}\) is a positive vector, satisfying the following conditions:
\[
\langle \vc{p}, \vc{c} \rangle = 0, \quad \langle \vc{p}, \vc{p} \rangle > 0.
\]

Let \(\mathbb{X}^{n-1}_{\vc{p}, \vc{c}}\) denote the hyperplane in \(\mathbb{X}^n\) that passes through \(\vc{c}\) and has \(\vc{p}\) as a polar vector. Since a polar vector uniquely determines a hyperplane, the notation \(\mathbb{X}^{n-1}_{\vc{p}, \vc{c}}\) is, in fact, redundant. However, to emphasize that this hyperplane passes through \(\vc{c}\), we retain the point \(\vc{c}\) in the notation.

\subsubsection{Tangent vector and slant angle of hyperplane}
Let $\vc{c}$ be a point in $\mathbb{X}^n$. The \emph{tangent plane} to $\mathbb{X}^n$ at $\vc{c}$ is the $n$-dimensional plane in $\mathbb{R}^{n+1}$ that is tangent to $\mathbb{X}^n$. This tangent plane is nothing other than the tangent space $\displaystyle T_{\vc{c}}\mathbb{X}^n\subset \mathbb{R}^{n+1}$. Tangent vectors can be regarded as vectors in the ambient space $\mathbb{R}^{n+1}$. If $\mathbb{X}^n=\mathbb{E}^n$, any tangent plane is the space $\mathbb{E}^n$ itself. It is clear that
\[
\displaystyle T_{\vc{c}}\mathbb{X}^n=
\begin{cases}
  \left\{\mathbb{X}\in\mathbb{R}^{n+1}\,|\, x_{n+1}=1\right\}, & \mbox{if } \mathbb{X}^n=\mathbb{E}^n, \\
  \left\{\mathbb{X}\in\mathbb{R}^{n+1}\,|\, \langle\mathbb{X},\vc{c}\rangle=\delta_{X}\right\}, & \mbox{if } \mathbb{X}^n=\mathbb{S}^n \text{ or }\mathbb{H}^n.
\end{cases}
\]

A \emph{tangent vector} $\vc{v}$ at $\vc{c}$ in the tangent space $T_{\vc{c}}\mathbb{X}^n$ is of the form $\vc{v}=\mathbb{X}-\vc{c}$ for some element $\mathbb{X}$ in the tangent plane at $\vc{c}$. \\
When $\mathbb{X}^n=\mathbb{S}^n \text{ or }\mathbb{H}^n$, it is clear that
\[
\vc{v}\in T_{\vc{c}}\mathbb{X}^n \quad\text{if and only if}\quad \langle\vc{v},\vc{c}\rangle=0.
\]
In particular, when $\mathbb{X}^n=\mathbb{S}^n \text{ or }\mathbb{H}^n$, a polar vector $\vc{p}$ of a hyperplane $\mathbb{X}^{n-1}_{\vc{p},\vc{c}}$ is regarded as a unit tangent vector in $T_{\vc{c}}\mathbb{X}^n$ and is regarded also as a unit normal vector to $T_{\vc{c}}\mathbb{X}^{n-1}_{\vc{p},\vc{c}}$.

Let $\vc{v}$ and $\vc{w}$ be tangent vectors in $T_{\vc{c}}\mathbb{X}^n$. The form $\langle~,~\rangle$ restricted to $T_{\vc{c}}\mathbb{X}^n$ is positive definite and the angle $\theta$ between $\vc{v}$ and $\vc{w}$ is given by
\[
\cos\theta=\frac{\langle\vc{v},\vc{w}\rangle} {\sqrt{\langle\vc{v},\vc{v}\rangle\langle\vc{w},\vc{w}\rangle}},
\quad 0\leq\theta\leq\pi.
\]

Consider a unit tangent vector $\vc{t}\in T_{\vc{c}}\mathbb{X}^n$ and a hyperplane $\mathbb{X}^{n-1}_{\vc{p},\vc{c}}$. The \emph{slant angle} of the hyperplane $\mathbb{X}^{n-1}_{\vc{p},\vc{c}}$ \emph{with respect to} $\vc{t}$ is defined as the angle between the vector $\vc{t}$ and the normal vector of $T_{\vc{c}}\mathbb{X}^n$.\\
When $\mathbb{X}^n=\mathbb{S}^n \text{ or }\mathbb{H}^n$, the slant angle $\theta$ satisfies that
\begin{equation}\label{f10}
\cos\theta
=\frac{\left|\langle\vc{t},\vc{p}\rangle\right|} {\sqrt{\langle\vc{t},\vc{t}\rangle\langle\vc{p},\vc{p}\rangle}}
=\left|\langle\vc{t},\vc{p}\rangle\right|,
\quad 0\leq\theta\leq\frac{\pi}{2}.
\end{equation}

Let $C$ be a regular curve in $\mathbb{X}^{n}$ . Consider a hyperplane that intersects $C$. We mean the \emph{slant angle} of this hyperplane \emph{with respect to the curve} $C$ by the slant angle of this hyperplane with respect to the unit tangent vector to this curve at the intersection point.

\subsubsection{Orthonormal basis and parametrization of hyperplane}
Consider a hyperplane $\mathbb{X}^{n-1}_{\vc{p},\vc{c}}$. Let the set of vectors $\{\vc{n}_1,\dots,\vc{n}_{n-1}\}$ be an orthonormal basis for the orthogonal complement $\text{Span}\{\vc{p},\vc{c}\}^{\perp}$ in $\mathbb{R}^{n+1}$. Then we can easily check the facts below.
\begin{itemize}
\item The set of vectors
\[
\{\vc{p},\vc{n}_1,\dots,\vc{n}_{n-1},\vc{c}\}
\]
is an orthonormal basis for the tangent space $T_{\vc{c}}\mathbb{R}^{n+1}\approx\mathbb{R}^{n+1}$.

\item All of the vectors $\vc{p},\vc{n}_1,\dots,\vc{n}_{n-1}$ are positive.

\item The set of vectors
\[
\{\vc{p},\vc{n}_1,\dots,\vc{n}_{n-1}\}
\]
is an orthonormal basis for the tangent space $T_{\vc{c}}\mathbb{X}^{n}$.

\item Since the hyperplane $\mathbb{X}^{n-1}_{\vc{p},\vc{c}}$ is the intersection $\vc{p}^{\perp}\cap\mathbb{X}^{n}$,
\begin{equation}\label{}
\mathbb{X}^{n-1}_{\vc{p},\vc{c}}=\text{Span}\{\vc{n}_1,\dots,\vc{n}_{n-1},\vc{c}\}\cap\mathbb{X}^{n}.
\end{equation}

\item The set of vectors
\[
\{\vc{n}_1,\dots,\vc{n}_{n-1}\}
\]
is an orthonormal basis for the tangent space $T_{\vc{c}}\mathbb{X}^{n-1}_{\vc{p},\vc{c}}$.
\end{itemize}

Any hyperplane $\mathbb{X}^{n-1}_{\vc{p},\vc{c}}\subset\mathbb{X}^{n}$ is isomorphic to $\mathbb{X}^{n-1}$. Making use of the parametrization $(x_1,\dots,x_{n})^T=\Phi^{n-1}(\varphi_1,\dots,\varphi_{n-1})$ of the space $\mathbb{X}^{n-1}\subset\mathbb{R}^n$, we obtain a parametrization of $\mathbb{X}^{n-1}_{\vc{p},\vc{c}}$:
\begin{equation}\label{eqn:X^{n-1}(p,c)}
  \vc{r}(\varphi_1,\dots,\varphi_{n-1})
  =\big(\vc{p}~~\vc{n}_1~\cdots~\vc{n}_{n-1}~~\vc{c}\big)
  \begin{bmatrix}
    0 \\ \Phi^{n-1}(\varphi_1,\dots,\varphi_{n-1})
  \end{bmatrix}.
\end{equation}
Let the matrix $F$ be defined as
\begin{equation}\label{}
  F:=\big(\vc{p}~~\vc{n}_1~\cdots~\vc{n}_{n-1}~~\vc{c}\big).
\end{equation}
Note that the matrix $F$ is an orthogonal matrix:
\begin{equation}\label{}
F^TJ_{n+1}F=J_{n+1}.
\end{equation}
The parametrization (\ref{eqn:X^{n-1}(p,c)}) is rewritten as
\begin{equation}\label{}
  \vc{r}(\varphi_1,\dots,\varphi_{n-1})
  =F\,\begin{bmatrix}
    0 \\ \Phi^{n-1}(\varphi_1,\dots,\varphi_{n-1})
  \end{bmatrix}.
\end{equation}

\section{Mass Centers of Finite Collection of Points}\label{section:mass nenters of finite collection of points}

\subsection{Material Points and Mass Center System}

See Figures \ref{efig1} and \ref{efig2} again. A \textbf{material point} in $\mathbb{X}^{n}$ is a point which has a non-negative mass. Let $\vc{r}$ be a material point in $\mathbb{X}^{n}$ with mass $m$. We represent this massed point by $(\vc{r},m)\in\mathbb{X}^{n}\times\mathbb R^{+}$.

Consider the correspondence between $(\vc{r},m)\in\mathbb{X}^{n}\times\mathbb R^{+}$ with $\mathbb{R}^{+}\mathbb{X}^{n}$ given by
\begin{equation}\label{1-1 correspondence}
  (\vc{r},m)\quad \longrightarrow\quad m\vc{r}.
\end{equation}
Using this correspondence, we represent the material point $(\vc{r},m)$ by the vector
\[
m\vc{r}\in\mathbb R^{+}\mathbb{X}^{n}.
\]
Conversely, the vector 
\[
\vc{a}=m_{\vc{a}}[\vc{a}]\in\mathbb{R}^{+}\mathbb{X}^{n}\setminus\{\vcz\}
\]
represents the material point $[\vc{a}]$ with mass $m_{\vc{a}}$. In this manner, we call each vector $\vc{a}$ in $\mathbb R^{+}\mathbb{X}^{n}$ a \textbf{material vector} representing the material point $[\vc{a}]$ with mass $m_{\vc{a}}$. By convention, the zero material vector $\vcz$ represents any point with \emph{zero mass}:
\begin{equation}
    \vcz=0[\vc{a}].
\end{equation}

By definition, a \textbf{mass center system for finite material points} on $\mathbb{X}^{n}$ is an assignment
\[
\mathcal{F}\,:\,\text{each of all finite collection of material points}\mapsto\text{a material point}
\]
or, equivalently,
\[
\mathcal{F}\,:\,\text{each of all finite collection of material vectors}\mapsto\text{a material vector}
\]
which satisfies the \emph{system of axioms} listed below: For notational convenience, we denote
\[
\mathcal{F}\{\vc{a}_1,\cdots,\vc{a}_N\}=:\vc{a}_1\oplus\cdots\oplus\vc{a}_N.
\]
Now, we list the \emph{system of axioms}.

\begin{itemize}

\item \emph{axiom of zero mass}: All points with zero mass are equal to one another.
\item \emph{axiom of single point}\footnote{Galperin called this axiom as the \emph{axiom of immovability}.}: For all material vector $\vc{a}$,
\[
\mathcal{F}(\vc{a})=\vc{a}.
\]
\item \emph{axiom of overlapping}\footnote{We altered the system of axioms from that of Galperin. Galperin did not list this axiom but listed instead the \emph{axiom of multiplication} which is  Proposition \ref{prop:Rule of multiplication} in this article.}: For all non-negative numbers $r_1, r_2$ and for all material vector $\vc{a}$,
\[
r_1\vc{a}\oplus r_2\vc{a}=(r_1+r_2)\vc{a}
\]
\item \emph{axiom of permutation and partition}: For all finite collection $\{\vc{a}_1,\cdots,\vc{a}_N\}$ of material vectors, for all integer $m$ with $1\leq m\leq N$, and for all permutation $\sigma$,
    \[
    \vc{a}_1\oplus\cdots\oplus\vc{a}_N
    =(\vc{a}_{\sigma(1)}\oplus\cdots\oplus\vc{a}_{\sigma(m)})
    \oplus(\vc{a}_{\sigma(m+1)}\oplus\cdots\oplus\vc{a}_{\sigma(N)}).
    \]
\item \emph{axiom of isometry invariance}: Let $g$ be an isometry on $\mathbb{X}^{n}$ and let $\vc{a}=m_{\vc{a}}[\vc{a}]$ be a material vector in $\mathbb{R}^{+}\mathbb{X}$. Define
    \[
    g(\vc{a})=m_{\vc{a}} g([\vc{a}]) \quad\text{and}\quad g(\vcz)=\vcz.
    \]
    Then, for all finite collection $\{\vc{a}_1,\cdots,\vc{a}_N\}$ of material vectors,
    \[
    g(\vc{a}_1)\oplus\cdots\oplus g(\vc{a}_N)
    =g(\vc{a}_1\oplus\cdots\oplus\vc{a}_N).
    \]


\item \emph{axiom of continuity}: Let $\vc{a}(t)$ and $\vc{b}(t)$ are continuously parametrized material vectors in $\mathbb{R}^{+}\mathbb{X}^{n}$ with respect to the usual topology on $\mathbb{R}^{n+1}$. Then the parametrized material vector
    \[
    \vc{a}(t)\oplus\vc{b}(t)
    \]
    is continuous as well.
\end{itemize}

We will show the uniqueness of the mass center system. By definition, we call the material vector
\[
\vc{a}_1\oplus\cdots\oplus\vc{a}_N
\]
the \textbf{mass center vector} of $\{\vc{a}_1,\,\dots,\,\vc{a}_N\}$ and we call the point
\[
[\vc{a}_1\oplus\cdots\oplus\vc{a}_N]
\]
a \textbf{mass center} of $\{\vc{a}_1,\,\dots,\,\vc{a}_N\}$. The mass of the mass center vector $\vc{a}_1\oplus\cdots\oplus\vc{a}_N$ is called, by definition,  the \textbf{centered mass} of $\{\vc{a}_1,\,\dots,\,\vc{a}_N\}$ whereas we call the sum of all masses the \textbf{total mass} of $\{\vc{a}_1,\,\dots,\,\vc{a}_N\}$.

Once the uniqueness is shown, it turns out that the notions such as mass center vector, mass center, and centered mass depend only on the relevant set $S$ and the geometry of $\mathbb{X}^n$.
We denote the centered mass and the total mass of the set $S=\{\vc{a}_1,\,\dots,\,\vc{a}_N\}$, respectively, as
\[
m_{\text{cen}}(S) \quad\text{and}\quad m_{\text{tot}}(S).
\]

In the rest of this chapter, we will show that there is a unique system of mass centers for each of $\mathbb{E}^n$, $\mathbb{S}^n$, and $\mathbb{H}^n$ when $n\geq2$.

\subsection{Basic Propositions}

\begin{prop}\label{identity}
\emph{(Rule of identity)}
 For all material vector $\vc a$,
\[
\vc{a}\oplus \vc {0}=\vc {0} \oplus \vc{a}=\vc{a}.
\]
\end{prop}
\begin{proof}
By the \emph{axiom of zero mass}, the zero vector $\vcz$ can be expressed as $0[\vc{a}]$. From the \emph{axiom of overlapping}, it follows that  
\[
\vc{a}\oplus \vc{0}=m_{\vc{a}}[\vc{a}]\oplus 0[\vc{a}]=(m_{\vc{a}}+0)[\vc{a}]=\vc{a}.
\]
Similarly, $\vc{0}\oplus \vc{a}=\vc{a}$ can be derived in the same manner.
\end{proof}
\begin{prop}\label{prop:Rule of multiplication}
\emph{(Rule of multiplication\footnote{In fact, the system of axioms that the \emph{axiom of overlapping} is substituted by the \emph{Rule of multiplication} is the same system to the former system.})}
 For all $r>0$,
\[
r\vc{a}\oplus r\vc{b}=r(\vc{a}\oplus\vc{b}).
\]
\end{prop}
\begin{proof}
By the \emph{axiom of continuity}, it suffices to consider the case where $r$ is a rational number. Let $r = \dfrac{l}{k}$, where $k$ and $l$ are positive integers.
Overlap the material vector $ \dfrac{l}{k}\vc{a}\oplus \dfrac{l}{k}\vc{b} $ $k$-times. Then, by the \emph{axiom of permutation and partition} and the \emph{axiom of overlapping}, we obtain $l(\vc{a}\oplus\vc{b})$
\begin{align*}
  k\left(\frac{l}{k}\vc{a}\oplus\frac{l}{k}\vc{b}\right)
&=\left(\frac{l}{k}\vc{a}\oplus\frac{l}{k}\vc{b}\right)
\oplus\dots\oplus\left(\frac{l}{k}\vc{a}\oplus\frac{l}{k}\vc{b}\right)
\\
&=\left(\frac{l}{k}\vc{a}\oplus\dots\oplus\frac{l}{k}\vc{a}\right)
\oplus\left(\frac{l}{k}\vc{b}\oplus\dots\oplus\frac{l}{k}\vc{b}\right)
\\
&=l\vc{a}\oplus l\vc{b}
\\
&=(\vc{a}\oplus\dots\oplus\vc{a})\oplus(\vc{b}\oplus\dots\oplus\vc{b})
\\
&=(\vc{a}\oplus\vc{b})\oplus\dots\oplus(\vc{a}\oplus\vc{b})\
\\
&=l(\vc{a}\oplus\vc{b}).
\end{align*}
Hence, we have
\[
\frac{l}{k}\vc{a}\oplus\frac{l}{k}\vc{b}=\dfrac{l}{k}(\vc{a}\oplus\vc{b}).
\]
\end{proof}

\begin{prop}]\label{prop33}
Let $\vc{a}$ and $\vc{b}$ be material vectors with positive masses. Then
\begin{equation*}
\vc{a}\oplus\vc{b}=\vcz \quad\text{if and only if}\quad\mathbb{X}^{n}=\mathbb{S}^n\ \text{and}\ \ \vc{a}+\vc{b}=\vcz.
\end{equation*}
\end{prop}
\begin{proof}
Assume that $\mathbb{X}^{n}=\mathbb{S}^n$ and $\vc{a}$, $\vc{b}$ are material vectors satisfying that $\vc{a}+\vc{b}=\vcz$. Then $\vc{a}=m[\vc{a}]$ and $\vc{b}=m(-[\vc{a}])$.
Let $\iota_0$ be the antipodal map on $\vc S^n$. It is clear that $\iota_0$ is an isometry on $\vc S^n$ and
\begin{align*}
&\iota_0(\vc{a})=\iota_0(m[\vc{a}])=m\iota_0([\vc{a}])=m(-[\vc{a}])=\vc{b},\\
&\iota_0(\vc{b})=\iota_0(m(-[\vc{a}]))=m\iota_0(-[\vc{a}])=m[\vc{a}]=\vc{a}.
\end{align*}
By the \emph{axiom of isometry invariance} and the \emph{axiom of permutation and partition}, we have
\[
\iota_0(\vc{a}\oplus\vc{b})=\iota_0(\vc{a})\oplus\iota_0(\vc{b})
=\vc{b}\oplus\vc{a}=\vc{a}\oplus\vc{b}.
\]
If the material point $\vc{a}\oplus\vc{b}$ has a positive mass, it cannot be invariant under $\iota_0$. Hence, we conclude that
\[
\vc{a}\oplus\vc{b}=\vcz.
\]
Conversely, assume that $\vc{a}$, $\vc{b}$ are material points satisfying that
\[
\vc{a}\oplus\vc{b}=\vcz.
\]
Let $\iota_{\vc{a}}$ be the involution in $[\vc{a}]$. By the \emph{Axiom of isometry invariance},
\[
\vc{a}\oplus\iota_{\vc{a}}(\vc{b})
=\iota_{\vc{a}}(\vc{a})\oplus\iota_{\vc{a}}(\vc{b})
=\iota_{\vc{a}}(\vc{a}\oplus\vc{b})
=\vcz.
\]
Then, by the \emph{axiom of permutation and partition},
\[
\vc{a}\oplus\vc{b}\oplus\iota_{\vc{a}}(\vc{b})
=\begin{cases}
   (\vc{a}\oplus\vc{b})\oplus\iota_{\vc{a}}(\vc{b})
   =\vcz\oplus\iota_{\vc{a}}(\vc{b})
   =\iota_{\vc{a}}(\vc{b}), \\
   \quad\text{\small{and simultaneously}} \\
   (\vc{a}\oplus\iota_{\vc{a}}(\vc{b}))\oplus\vc{b}
   =\vcz\oplus\vc{b}
   =\vc{b}.
 \end{cases}
\]
Thus $\iota_{\vc{a}}(\vc{b})=\vc{b}$. The point $[\vc{b}]$ is fixed under the involution $\iota_{\vc{a}}$. Hence, the point $[\vc{b}]$ is either $[\vc{a}]$ or, with $\mathbb{X}^{n}=\vc S^n$, the antipodal point of $[\vc{a}]$. If $[\vc{b}]=[\vc{a}]$, then $\vc{a}\oplus\vc{b}=(m_{\vc{a}}+m_{\vc{b}})[\vc{a}]\neq\vcz$. This is a contradiction. Hence, we have $\mathbb{X}^n=\mathbb{S}^n$, $[\vc{b}]=[-\vc{a}]$, and we can get the following with a natural assumption $m_{\vc{a}}\ge m_{\vc{b}}$:
\[
\vc{a}\oplus\vc{b}=(m_{\vc{a}}-m_{\vc{b}})[\vc{a}]\oplus m_{\vc{b}}[\vc{a}]\oplus m_{\vc{b}}[\vc{b}]=(m_{\vc{a}}-m_{\vc{b}})[\vc{a}]\oplus\vc{0}=(m_{\vc{a}}-m_{\vc{b}})[\vc{a}]=\vcz.
\]
Therefore we obtain $m_{\vc{a}}=m_{\vc{b}}$ and conclude that $\vc{a}+\vc{b}=\vcz$.
\end{proof}

For spherical spaces, the concept of a midpoint requires clarification, as the geodesic segment connecting two points is not unique. For a pair of antipodal points, we define that the midpoint does not exist. For non-antipodal points, there are two geodesic segments connecting them: one shorter and one longer. By definition, the midpoint is taken to be the point on the shorter segment. From this point onward, we will refer to the geodesic segment as the shorter one.

It can be easily verified that if $\vc{a} = m[\vc{a}]$ and $\vc{b} = m[\vc{b}]$ are material vectors with equal mass, then the midpoint between $[\vc{a}]$ and $[\vc{b}]$ is given by 
\[
[\vc{a} + \vc{b}]
\]
using formulas (\ref{distance1}) and (\ref{distance2}).
\begin{prop}\label{prop:mass center is on segment}
Let $\vc{a}$ and $\vc{b}$ be material vectors such that $\vc{a}+\vc{b}\neq\vcz$.
\begin{itemize}
\item
Then, the mass center $[\vc{a} \oplus \vc{b}]$ lies on the geodesic in $\mathbb{X}^{n}$ passing through $[\vc{a}]$ and $[\vc{b}]$. Therefore, we have
\[
\vc{a} \oplus \vc{b} \in \text{\emph{Span}}\{\vc{a}, \vc{b}\}.
\]
\item
In particular, if $m(\vc{a})=m(\vc{b})$, the mass center of them is the mid-point of them:
\[
[\vc{a}\oplus\vc{b}]=[\vc{a}+\vc{b}].
\]
\end{itemize}
\end{prop}
\begin{proof}
Since the given material points remain fixed under the involution in the geodesic passing through them, the mass center must lie on that geodesic.

Since the given material points are fixed under the involution in the geodesic through them, the mass center must be on that geodesic.

Assume that the masses of two material points are same and let $\iota$ be the involution in the mid-point $[\vc{a}+\vc{b}]$ of them. Then, it is clear that $\iota(\vc{a})=\vc{b}$ and $\iota(\vc{b})=\vc{a}$. We have
\[
\iota(\vc{a}\oplus\vc{b})=\iota(\vc{a})\oplus\iota(\vc{b})
=\vc{b}\oplus\vc{a}=\vc{a}\oplus\vc{b}.
\]
The point $[\vc{a} \oplus \vc{b}]$ is fixed under the involution $\iota$. 

When $\mathbb{X}^{n} = \mathbb{E}^n$ or $\mathbb{X}^{n} = \mathbb{H}^n$, there is a unique fixed point under the involution $\iota$. Therefore, we obtain that
\[
[\vc{a} \oplus \vc{b}] = [\vc{a} + \vc{b}].
\]

When $\mathbb{X}^{n} = \mathbb{S}^n$, there are two possibilities. The point $[\vc{a} \oplus \vc{b}]$ is either on the short geodesic segment $I$ joining $[\vc{a}]$ and $[\vc{b}]$ or on the segment $\tilde{I}$, which is the geodesic segment antipodal to $I$. Note that $I$ and $\tilde{I}$ are disjoint, and
\[
[\vc{a} \oplus \vc{b}] =
\begin{cases}
[\vc{a} + \vc{b}], & \text{if } [\vc{a} \oplus \vc{b}] \in I, \\
[-\vc{a} - \vc{b}], & \text{if } [\vc{a} \oplus \vc{b}] \in \tilde{I}.
\end{cases}
\]
Let $\vc{b}_t$ be a continuously parametrized material vector satisfying that
\begin{itemize}
  \item $[\vc{b}_t]\in I$ for all $t\in[0,1]$,
  \item $m_{\vc{a}}=m_{\vc{b}}=m_{\vc{b}_t}$, for all $t\in[0,1]$,
  \item $\vc{b}_0=\vc{a}$ and $\vc{b}_1=\vc{b}$.
\end{itemize}
By the \emph{axiom of continuity}, the point $[\vc{a}\oplus\vc{b}_t]$ is either on $I$ for all $t\in[0,1]$ or on $\tilde{I}$ for all $t\in[0,1]$. Since, by the \emph{axiom of overlapping},
\[
[\vc{a}\oplus\vc{b}_0]=[\vc{a}\oplus\vc{a}]=[2\vc{a}]=[\vc{a}]\in I,
\]
we can conclude that
\[
[\vc{a}\oplus\vc{b}_1]=[\vc{a}\oplus\vc{b}]\in I.
\]
This is the end of the proof.
\end{proof}


\subsection{Existence of mass center system for  multi-dimensional space }
\begin{prop}
There is a mass center system for finite collection of material point as stated below :
\[
\vc{a}_1\oplus\dots\oplus\vc{a}_N=\vc{a}_1+\dots+\vc{a}_N.
\]
\end{prop}

\begin{proof}
Apart from the \textit{axiom of isometry invariance}, it is evident that this system satisfies the other axioms. We know that the isometry group of $\mathbb{S}^n$ and $\mathbb{H}^n$ are, respectively, the groups $\text{O}(n+1)$ and $\text{O}(n,1)$. These groups act not only on the considered spaces but also on the ambient space $\mathbb{R}^{n+1}$. Each element in these isometry groups is a linear map. The isometry group of $\mathbb{E}^n$ is the semidirect product $O(n) \ltimes \mathbb{R}^n$ and each element in this group is a linear map represented by the matrix of the form 
\[
\left(
\begin{array}{c|c}
  A & \vc{b} \\
  \hline
  \vcz & 1
\end{array}
\right),
\quad A\in O(n),~\vc{b}\in\mathbb{R}^n.
\]
Hence, the system described in this proposition satisfies the \textit{axiom of isometry invariance}.
\end{proof}

\subsection{Uniqueness of mass center system for  multi-dimensional space }
Let us first consider the case where the dimension is greater than or equal to two. In this subsection, we state and prove the uniqueness of the mass center system for dimensions two and higher. The following lemma plays a crucial role in proving the uniqueness theorem.

\begin{lemma}\label{lem:integer mass, [oplus]=[+]}
Assume that $\dim\mathbb{X}^{n}\geq2$. Let $\vc{u},\vc{v}$ be material vectors with $m_{\vc{u}}=m_{\vc{v}}$.  Then, for all positive integers $k$ and $l$,
\begin{equation}\label{eqn1 in lem:integer mass, [oplus]=[+]}
[k\vc{u}\oplus l\vc{v}]=[k\vc{u}+l\vc{v}].
\end{equation}
\end{lemma}
\begin{proof}
For the case that $\vc{u}=\vc{v}$, this lemma is obviously true.

For the case that $\mathbb{X}^{n}=\mathbb{S}^n$ and $\vc{u}+\vc{v}=\vcz$, this lemma is true as well: We may assume that $k\leq l$. Observe that
\begin{align*}
[k\vc{u}\oplus l\vc{v}]&=[k\vc{u}\oplus l(-\vc{u})]
\\
&=[k\{\vc{u}\oplus(-\vc{u})\}\oplus(l-k)(-\vc{u})]
\\
&=[\vcz\oplus(l-k)(-\vc{u})]
\\
&=[(k-l)\vc{u}]
\\
&=[k\vc{u}+l\vc{v}].
\end{align*}

For the case that $\vc{u}$ and $\vc{v}$ are linearly independent, we use the weak induction on $k+l$.

By the second assertion of Proposition \ref{prop:mass center is on segment}, this lemma is true for $k=l=1$.\\
Assume that this lemma is true for $k+l\leq N$ and let $p$, $q$ be any pair of positive integers such that $p+q=N$. To complete the induction argument, it is enough to show that
\[
[(p+1)\vc{u}\oplus q\vc{v}]=[(p+1)\vc{u}+q\vc{v}]
\]
and
\[
[p\vc{u}\oplus(q+1)\vc{v}]=[p\vc{u}+(q+1)\vc{v}].
\]

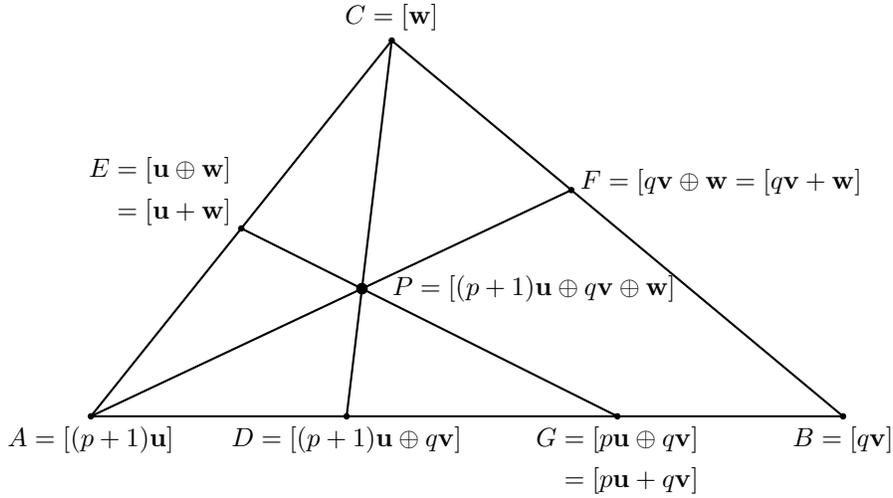
\begin{figure}[h!]
    \vspace{1cm}
		\begin{center}
   \input{triangle3.tex}
   \caption{concurrent segments $\overline{AF}$, $\overline{CD}$, $\overline{EG}$ and the induction hypothesis applied on the points $E$, $F$, $G$  }\label{fig:triangleabc}
		\end{center}
	\end{figure}
    
Since $\dim\mathbb{X}^{n}\geq2$, we can choose a material vector $\vc{w}$ such that the vectors $\vc{u}$, $\vc{v}$, and $\vc{w}$ are linearly independent and that $m_{\vc{u}}=m_{\vc{v}}=m_{\vc{w}}$.\
Let 
\begin{equation}
    A=[(p+1)\vc{u}],~~B=[q\vc{v}],~~C=[\vc{w}],
\end{equation}
and let 
\begin{equation}
    D=[(p+1)\vc{u}\oplus q\vc{v}], ~~E=[\vc{u}\oplus\vc{w}],~~F=[q\vc{v}\oplus\vc{w}], ~~G=[p\vc{u}\oplus q\vc{v}],
\end{equation}
Then, as seen in Figure \ref{fig:triangleabc},
\begin{equation}
    D,G\in\overline{AB}, ~~E\in\overline{CA}, ~~F\in\overline{BC}.
\end{equation}
By the \emph{axiom of permutation and partition} and the \emph{axiom of overlapping}, three segments $\overline{AF}$, $\overline{CD}$, $\overline{EG}$ are concurrent at the point
\begin{equation*}
  P=[(p+1)\vc{u}\oplus q\vc{v}\oplus\vc{w}].
\end{equation*}
By the induction hypothesis, we have
\begin{equation*}
  E=[\vc{u}+\vc{w}], 
  ~~F=[q\vc{v}+\vc{w}], 
  ~~G=[p\vc{u}+q\vc{v}].
\end{equation*}
From Proposition \ref{prop:mass center is on segment} and the fact that the segments  $\overline{AF}$ and $\overline{EG}$ meet at $P$, we can observe 
\begin{align*}
(p+1)\vc{u}\oplus q\vc{v}\oplus\vc{w}
& \,\in\,\text{Span}\{(p+1)\vc{u}\,,\,q\vc{v}+\vc{w}\}
\cap\text{Span}\{\vc{u}+\vc{w},p\vc{u}+q\vc{v}\}
\\
&\,=\, \text{Span}\{(p+1)\vc{u}+q\vc{v}+\vc{w}\}.
\end{align*}
From Proposition \ref{prop:mass center is on segment} and the fact that the segments  $\overline{AB}$ and the geodesic through $P$ and $C$ meet at $D$, we can observe 
\begin{align*}
(p+1)\vc{u}\oplus q\vc{v}
&\,\in \,
\text{Span}\{(p+1)\vc{u}\,,\,q\vc{v}\}
\cap\text{Span}\{(p+1)\vc{u}+q\vc{v}+\vc{w},\vc{w}\}
\\
& \,=\, \text{Span}\{(p+1)\vc{u}+q\vc{v}\}.
\end{align*}
Hence we have either
\[
[(p+1)\vc{u}\oplus q\vc{v}]=[(p+1)\vc{u}+q\vc{v}]
\]
or
\[
[(p+1)\vc{u}\oplus q\vc{v}]=[-(p+1)\vc{u}-q\vc{v}].
\]

When $\mathbb{X}^{n}=\mathbb{E}^n$ or $\mathbb{X}^{n}=\mathbb{H}^n$, the vector $-(p+1)\vc{u}-q\vc{v}$ cannot be a positive multiple of a material vector on the underlying space. Hence, it must be true that $[(p+1)\vc{u}\oplus q\vc{v}]=[(p+1)\vc{u}+q\vc{v}]$.

When $\mathbb{X}^{n}=\mathbb{S}^n$, the point $[(p+1)\vc{u}+q\vc{v}]$ is on the geodesic segment $I$ joining $[\vc{u}]$ and $[\vc{v}]$. The point $[-(p+1)\vc{u}-q\vc{v}]$ is on the  geodesic segment which is antipodal to the geodesic segment $I$. When $[\vc{u}]=[\vc{v}]$, we see
\[
[(p+1)\vc{u}\oplus q\vc{v}]=[\vc{u}]\in I.
\]
By a similar argument used in the proof of Proposition \ref{prop:mass center is on segment}, the point $[(p+1)\vc{u}\oplus q\vc{v}]$ must be on $I$. Hence we can conclude that
\[
[(p+1)\vc{u}\oplus q\vc{v}]=[(p+1)\vc{u}+q\vc{v}],
\]
for the case $\mathbb{X}^{n}=\mathbb{S}^n$ as well.

We omit the proof for the fact that $[p\vc{u}\oplus(q+1)\vc{v}]=[p\vc{u}+(q+1)\vc{v}]$.
This completes the induction arguments and finishes the proof.
\end{proof}

\begin{theorem}\emph{(Uniqueness of the mass center system for finite collections of material points)}\label{unique}
 Assume that $\dim{X}^{n}\geq2$, then there is a unique mass center system: For material vectors $\vc{a}_1,\dots,\vc{a}_N$, we have
\begin{equation}\label{eqn: the mass center system} \vc{a}_1\oplus\dots\oplus\vc{a}_N=\vc{a}_1+\dots+\vc{a}_N.
\end{equation}
\end{theorem}
\begin{proof}
We will prove the followings.
\begin{itemize}
    \item For all material vectors $\vc{a}$ and $\vc{b}$, we have
    \begin{equation}\label{f26}
    [\vc{a}\oplus\vc{b}]=[\vc{a}+\vc{b}].
    \end{equation}
    \item For all material vectors $\vc{a}$ and $\vc{b}$, we have
    \begin{equation}\label{f28}
    \vc{a}\oplus\vc{b}=\vc{a}+\vc{b}.
    \end{equation}
\end{itemize}

We now prove the first statement. \\
Let \( \vc{b} = r\vc{b}' \) for some positive real number \( r \) and a material vector \( \vc{b}' \) whose mass is equal to the mass of \( \vc{a} \). We claim that the first statement is true when \( r = \dfrac{l}{k} \) is a rational number. By the \emph{rule of multiplication} and Equation (\ref{eqn1 in lem:integer mass, [oplus]=[+]}) in Lemma \ref{lem:integer mass, [oplus]=[+]}, we have
\begin{equation*}
\left[\vc{a} \oplus \vc{b}\right]
= \left[\vc{a} \oplus \frac{l}{k} \vc{b}'\right]
= \left[k \vc{a} \oplus l \vc{b}'\right]
= \left[k \vc{a} + l \vc{b}'\right]
= \left[\vc{a} + \frac{l}{k} \vc{b}'\right]
= \left[\vc{a} + \vc{b}\right].
\end{equation*}
By the \emph{axiom of continuity}, Equation (\ref{f26}) holds for any positive real number \( r \).

We now proceed to prove the second statement. \\
By the \emph{Axiom of Overlapping}, this statement is obviously true when \( \vc{a} \) and \( \vc{b} \) are linearly dependent. Assume they are linearly independent. Choose a material vector \( \vc{c} \) so that the vectors \( \vc{a} \), \( \vc{b} \), and \( \vc{c} \) are linearly independent.

From Equation (\ref{f26}), we see
\[
  \vc{a} \oplus \vc{b} \oplus \vc{c} = 
  \begin{cases}
     (\vc{a} \oplus \vc{b}) \oplus \vc{c} = k_1 (\vc{a} + \vc{b}) \oplus \vc{c} = k_2 \{ k_1 (\vc{a} + \vc{b}) + \vc{c} \}, \\
     \quad \text{\small{and simultaneously}} \\
     \vc{a} \oplus (\vc{b} \oplus \vc{c}) = \vc{a} \oplus l_1(\vc{b} + \vc{c}) = l_2 \{ \vc{a} + l_1(\vc{b} + \vc{c}) \}.
  \end{cases}
\]
Hence,
\[
k_2 \{ k_1 (\vc{a} + \vc{b}) + \vc{c} \} = l_2 \{ \vc{a} + l_1(\vc{b} + \vc{c}) \}.
\]
This vector does not vanish by the linear independence of the vectors \( \vc{a} \), \( \vc{b} \), and \( \vc{c} \) and Proposition \ref{prop33}. Also  by the linear independence of the vectors \( \vc{a} \), \( \vc{b} \), and \( \vc{c} \), we conclude that \( k_1 = 1 \), and
\begin{equation}\label{f29}
 \vc{a} \oplus \vc{b} = \vc{a} + \vc{b}.
\end{equation}

Applying Equation (\ref{f29}) inductively, we obtain the conclusion of this theorem.
\end{proof}

\subsection{Uniqueness Problem for  one dimensional space}
In this section, we investigate the uniqueness problem of the mass center system on a one-dimensional metric space without boundary. To state the conclusion, there is a unique mass center system on \( \mathbb{S}^1 \) with length \( 2\pi \), and there are uncountably many mass center systems on a one-dimensional manifold with a complete metric. We will demonstrate these facts.

\subsubsection{Mass center system on a circle}

Let \( \vc{u} \) and \( \vc{v} \) be material vectors \emph{with the same mass}. The mass center \( [\vc{u} \oplus \vc{v}] \) is the midpoint of \( [\vc{u}] \) and \( [\vc{v}] \). Hence, we have \( \vc{u} \oplus \vc{v} = k(\vc{u} + \vc{v}) \), where \( k \) is a number. It is easy to verify that this number \( k \) does not depend on the mass of \( \vc{u} \) or \( \vc{v} \), but possibly depends only on the distance \( d \) between \( [\vc{u}] \) and \( [\vc{v}] \).

We introduce the function \( f: [0, \pi] \rightarrow \mathbb{R}^+ \) defined by
\[
\vc{u} \oplus \vc{v} = f(d)(\vc{u} + \vc{v}).
\]
Clearly, this function \( f \) is continuous and
\[ 
f(0) = 1 .
\] 
Our goal now is to show that \( f \) is a constant function.

\begin{lemma}\label{s1}
The function \( f \) above has the properties below:
\begin{equation}\label{lem-13}
  f(\pi - d) = f\left( d/2 \right) f\left( \pi - d/2 \right),
\end{equation}
\begin{equation}\label{lem-14}
  f\left( \pi/2 \right) = 1,
\end{equation}
\begin{equation}\label{lem-15}
  f\left( d/2 \right) = 1 \quad \text{if and only if} \quad f(d) = 1.
\end{equation}
\end{lemma}

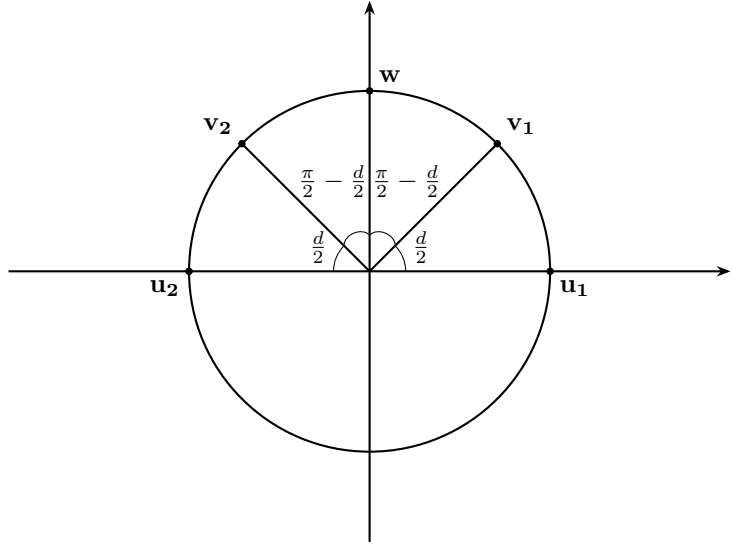
\begin{figure}[h!]
		\begin{center}
   \input{circle1.tex}
   \caption{five material points on $\mathbb{S}^1$}\label{fig:circle77}
		\end{center}
	\end{figure}

\begin{proof}
In figure \ref{fig:circle77}, there are five material vectors
\[
\vc{u}_1=1, ~\vc{u}_2=e^{i\pi},~\vc{v}_1=e^{i\frac{d}{2}}, ~\vc{v}_2=e^{i\left(\pi-\frac{d}{2}\right)},~\vc{w}=e^{i\frac{\pi}{2}}.
\]
To prove the first equation, take the vectors $\vc{u}_1, ~\vc{v}_1,~\vc{u}_2, ~\vc{v}_2$ into consideration and observe that
 \begin{align*}
     \vc{u}_1\oplus\vc{v}_1 \oplus  \vc{u}_2\oplus\vc{v}_2
     &=f\left(d/2\right)(\vc{u}_1+\vc{v}_1)\oplus f\left(d/2\right)(\vc{u}_2+\vc{v}_2)\\
     &=f\left(d/2\right)\left\{(\vc{u}_1+\vc{v}_1)\oplus (\vc{u}_2+\vc{v}_2)\right\}\\
     &=f\left(d/2\right) f\left(\pi-d/2\right) \{(\vc{u}_1+\vc{v}_1)+(\vc{u}_2+\vc{v}_2)\}\\
     &=f\left(d/2\right) f\left(\pi-d/2\right) (\vc{v}_1+\vc{v}_2).\\
     \end{align*}
On  the other hand
\begin{align*}
     \vc{u}_1\oplus\vc{v}_1 \oplus  \vc{u}_2\oplus\vc{v}_2
     &=(\vc{u}_1\oplus\vc{u}_2) \oplus  \vc{v}_1\oplus\vc{v}_2\\
     &=\vcz \oplus  \vc{v}_1\oplus\vc{v}_2\\
     &=f\left(\pi-d\right)( \vc{v}_1+ \vc{v}_2).\\
     \end{align*}
Since $\vc{v}_1+ \vc{v}_2\neq\vcz$,  Equation (\ref{lem-13}) holds. Put $d=\pi$ into Equation (\ref{lem-13}) to see that $\displaystyle f(0)=f\left(\frac{\pi}{2}\right)^2$ and hence that $\displaystyle f\left(\frac{\pi}{2}\right)=1$. 

To prove the last statement (\ref{lem-15}), take the vectors $\vc{v}_1, ~\vc{v}_2,~\vc{w}$ into consideration and observe that
 \begin{align*}
     \vc{v}_1\oplus \vc{w}\oplus  \vc{w}\oplus \vc{v}_2
     &=f\left(\pi/2-d/2\right)( \vc{v}_1+ \vc{w})
     \oplus f\left(\pi/2-d/2\right)( \vc{w}+ \vc{v}_2)\\
     &=f\left(\pi/2-d/2\right)\left\{(\vc{v}_1+ \vc{w})\oplus(\vc{w}+ \vc{v}_2)\right\}\\
     &=f\left(\pi/2-d/2\right)^2\left(\vc{v}_1+ 2\vc{w}+\vc{v}_2\right).
 \end{align*}
On  the other hand
\begin{align*}
     &\vc{v}_1\oplus \vc{w}\oplus  \vc{w}\oplus \vc{v}_2\\
     =&f\left(\pi-d\right)( \vc{v}_1+ \vc{v}_2)\oplus 2 \vc{w}\\
     = &f\left(\pi-d\right)( \vc{v}_1+ \vc{v}_2)+ 2\vc{w}.
     \quad (\text{since }[\vc{v}_1+ \vc{v}_2] \text{ and } [\vc{w}] \text{ overlap}.)
 \end{align*}
 Hence, we have $f\left(\pi/2-d/2\right)^2\left(\vc{v}_1+ 2\vc{w}+\vc{v}_2\right)
 =f\left(\pi-d\right)( \vc{v}_1+ \vc{v}_2)+ 2\vc{w}$. From this equation, it is clear that $f\left(\pi/2-d/2\right)=1$ if and only if $f\left(\pi-d\right)=1$ and this is equivalent to the fact that $f\left(d/2\right)=1$ if and only if $f\left(d\right)=1$.
 \end{proof}

Now, we introduce the set of dyadic numbers
\[
D_k=\left\{ \frac{m\pi}{2^k} ~|~ m=0,1,2,\dots, 2^k  \right\}.
\]
The set $D_k$ is evenly distributed in the interval $[0,\pi]$ and $d\in D_k$ if and only if $\pi-d \in D_k$. The union $\displaystyle\cup_{k=0}^{\infty}D_k$ is dense in $[0,\pi]$.
\begin{lemma}  \label{s1-2}
    
    For all $x$, f(x)=1, and consequently, for every pair of material vectors $\vc{u}$ and $\vc{v}$ with same mass,
     \[
     \vc{u}\oplus\vc{v}=\vc{u}+\vc{v}.
     \]
     
\end{lemma}
\begin{proof} Firstly, we are to show that $f(x)=1$ for all $x\in\displaystyle\cup_{k=0}^{\infty}D_k$ using an induction on $k$. \\
By Equation  (\ref{lem-14}), $f(d)=1$ for all $d\in D_1=\{0,\pi/2, \pi\}$.  Assume that $f(d)=1$ for all $d\in D_k$. By the fact that $d\in D_k$ if and only if $\pi-d\in D_{k}$, (\ref{lem-15}) and Equation (\ref{lem-13}), we get
\[
 f(\pi-d)=1,
\quad f\left(d/2\right)=1,
\quad f\left(\pi-d/2\right)=1.
\]
Let $x\in D_{k+1}$. Since
\[
x=
\begin{cases}
  d/2 ~\text{ for some } d\in D_k,  & \mbox{if } x\leq\dfrac{\pi}{2}, \\
  \pi-d/2 ~\text{ for some } d\in D_k,  & \mbox{otherwise},
\end{cases}
\]
we can see that $f(x)=1$ for all $x\in D_{k+1}$ and finish the induction argument. Since the set $\displaystyle\bigcup_{k=0}^{\infty}D_k$ is dense and the function $f$ is continuous, the function $f$ is identically $1$ in $[0,\pi]$.
\end{proof}

\begin{theorem}\label{uniques1}
There is a unique mass center system for \( \mathbb{S}^1 \). Specifically, for material vectors \( \vc{a}_1, \dots, \vc{a}_N \), we have
\[
\vc{a}_1 \oplus \dots \oplus \vc{a}_N = \vc{a}_1 + \dots + \vc{a}_N.
\]
\end{theorem}

\begin{proof}
   Let \(\vc{a}\) and \(\vc{b}\) be material vectors with masses that are not necessarily the same. We construct two sequences \(\{\vc{a}_i\}\) and \(\{\vc{b}_i\}\) of material vectors following the \emph{split and merge process} described below.

\begin{itemize}
    \item Set \(\vc{a}_1 = \vc{a}\) and \(\vc{b}_1 = \vc{b}\).
    \item If \(\vc{a}_1\) is heavier than \(\vc{b}_1\), split \(\vc{a}_1\) as \(\vc{a}_1 = (1-k)\vc{a}_1 + k\vc{a}_1\) \((0 < k < 1)\) such that \(m_{k\vc{a}_1} = m_{\vc{b}_1}\), and merge \(k\vc{a}_1\) and \(\vc{b}_1\). Set
    \[
    \vc{a}_2 = (1-k)\vc{a}_1, \quad \vc{b}_2 = k\vc{a}_1 \oplus \vc{b}_1 = k\vc{a}_1 + \vc{b}_1.
    \]
    It is obvious that
    \begin{equation}\label{sm1}
        \vc{a}_2 \oplus \vc{b}_2 = \vc{a}_1 \oplus \vc{b}_1, \quad \vc{a}_2 + \vc{b}_2 = \vc{a}_1 + \vc{b}_1,
    \end{equation}
    and since \(\vc{b}_2\) is the midpoint of \(\vc{a}_1\) and \(\vc{b}_1\), we get
    \begin{equation}\label{sm2}
        \overline{\vc{a}_2 \vc{b}_2} \subset \overline{\vc{a}_1 \vc{b}_1}, \quad
        d(\vc{a}_2, \vc{b}_2) = \frac{1}{2}d(\vc{a}_1, \vc{b}_1).
    \end{equation}
    \item If \(\vc{b}_1\) is heavier than \(\vc{a}_1\), execute a similar process to construct \(\vc{a}_2\) and \(\vc{b}_2\). Whichever of \(\vc{a}_1\) or \(\vc{b}_1\) is heavier, we can obtain the same formulas (\ref{sm1}) and (\ref{sm2}).
    \item We can repeat the \emph{split and merge process} to construct \(\vc{a}_3\) and \(\vc{b}_3\) from \(\vc{a}_2\) and \(\vc{b}_2\).
    \item Continue this process indefinitely.
    \item If \(m_{\vc{a}_N} = m_{\vc{b}_N}\) for some \(N \geq 1\) for the first time, stop the process. In this case, we easily obtain that \(\vc{a} \oplus \vc{b} = \vc{a} + \vc{b}\). Therefore, we need to consider only the non-stopping case, i.e., infinite sequences \(\{\vc{a}_i\}\) and \(\{\vc{b}_i\}\).
\end{itemize}

In summary, we have constructed the sequences \(\{\vc{a}_i\}\) and \(\{\vc{b}_i\}\) satisfying the following properties:
\begin{align*}
    &\vc{a}_1 = \vc{a}, \quad \vc{b}_1 = \vc{b}, \\
    &\vc{a} \oplus \vc{b} = \vc{a}_i \oplus \vc{b}_i, \quad \forall i \geq 1, \\
    &\vc{a} + \vc{b} = \vc{a}_i + \vc{b}_i, \quad \forall i \geq 1, \\
    &\overline{\vc{a} \vc{b}} = \overline{\vc{a}_1 \vc{b}_1} \supset \overline{\vc{a}_2 \vc{b}_2} \supset \overline{\vc{a}_3 \vc{b}_3} \supset \cdots, \\
    &d(\vc{a}_i, \vc{b}_i) = \frac{1}{2^{i-1}} d(\vc{a}, \vc{b}).
\end{align*}
    
We now aim to show that there exists a pair of convergent subsequences, \(\{\vc{a}_{i_j}\}\) and \(\{\vc{b}_{i_j}\}\). To do so, we alter our perspective and treat the material vectors \(\vc{a}_i\) and \(\vc{b}_i\) as vectors in the ambient space \(\mathbb{R}^2\). Observe that the parallelogram formed by the vectors \(\vc{a}_{i+1}\) and \(\vc{b}_{i+1}\) is contained within the parallelogram formed by the vectors \(\vc{a}_i\) and \(\vc{b}_i\) for all \(i\). Therefore, every vector in the sequences \(\{\vc{a}_i\}\) and \(\{\vc{b}_i\}\) is contained within the parallelogram formed by the vectors \(\vc{a}\) and \(\vc{b}\), which is a compact set.

By the \emph{Bolzano-Weierstrass theorem}, we know that the sequence of pairs \(\{(\vc{a}_i, \vc{b}_i)\}\) has a convergent subsequence, say \(\{(\vc{a}_{i_j}, \vc{b}_{i_j})\}\). Using the \emph{axiom of continuity}, we can then observe that
\[
\vc{a} \oplus \vc{b} = \lim_{j \to \infty} (\vc{a}_{i_j} \oplus \vc{b}_{i_j}) = \lim_{j \to \infty} \vc{a}_{i_j} \oplus \lim_{j \to \infty} \vc{b}_{i_j}.
\]
Similarly, it is trivial to show that
\[
\vc{a} + \vc{b} = \lim_{j \to \infty} (\vc{a}_{i_j} + \vc{b}_{i_j}) = \lim_{j \to \infty} \vc{a}_{i_j} + \lim_{j \to \infty} \vc{b}_{i_j}.
\]
Applying the \emph{Nested Intervals theorem}, we conclude that the sequences \(\{[\vc{a}_{i_j}]\}\) and \(\{[\vc{b}_{i_j}]\}\) converge to the same point in \(\mathbb{S}^1\). Therefore, the material points \(\displaystyle\lim_{j \to \infty} \vc{a}_{i_j}\) and \(\displaystyle\lim_{j \to \infty} \vc{b}_{i_j}\) overlap.

By the \emph{axiom of overlapping}, we then have
\[
\lim_{j \to \infty} \vc{a}_{i_j} \oplus \lim_{j \to \infty} \vc{b}_{i_j} = \lim_{j \to \infty} \vc{a}_{i_j} + \lim_{j \to \infty} \vc{b}_{i_j}.
\]
Thus, we obtain the equation
\[
\vc{a} \oplus \vc{b} = \vc{a} + \vc{b}.
\]
By induction, we conclude the result of this theorem.
\end{proof}

\subsubsection{Mass center system on a line}
Now we consider the uniqueness problem of mass center system for each of $\mathbb{E}^1$, $\mathbb{S}^1$, and $\mathbb{H}^1$. Surprisingly, the uniqueness of mass center system is not satisfied in $\mathbb{E}^1$ and $\mathbb{H}^1$, and there are uncountably many mass center systems under our 6 mass center axioms. Therefore we know that
the dimension condition $\dim\mathbb{X}^{n}\geq2$ is very crucial to the uniqueness problem.

\begin{theorem} There are uncountably many mass center systems on $\mathbb{E}^1$ and $\mathbb{H}^1$ satisfying the six mass center axioms.
\end{theorem}
\begin{proof}For $N$ material vectors $\{m_i(x_i,1)\}_{i=1}^{N}$ in $\mathbb{E}^1$, we can define $\mathcal{F}_{k}\{m_i(x_i,1)\}_{i=1}^{N}$ $=\mathfrak{M}_k(\mathfrak{X}_k,1)$
for $k\in[0,\infty)$, where
\begin{align*}
\mathfrak{M}_k &=\begin{cases}\sqrt{\sum_{i=1}^N m_i^2+2\sum_{i\ne j}m_im_j\cosh k(x_i-x_j)},&\text{if } k\ne 0,\\\sum_{i=1}^N m_i,\quad\qquad&\text{if }k=0,\end{cases}\\
\mathfrak{X}_k &=\begin{cases}\frac{1}{k}\sinh^{-1}\left(\frac{1}{\mathfrak{M}_k}{\sum_{i=1}^N m_i\sinh(kx_i)}\right),&\text{if } k\ne 0,\\
    \frac{1}{\mathfrak{M}_0}\sum_{i=1}^N m_i x_i,&\text{if }k=0.\end{cases}
\end{align*}
 Then the mass center system $\mathcal{F}_{k}$ in $\mathbb{E}^1$ easily satisfies the all axioms except the \emph{axiom of isometry invariance}. And that is sufficient to prove that the system is invariant under the reflection $g_1(x)=-x$ and the translation $g_2(x)=x+p$.

 Hence we have to show that
 $$\mathcal{F}_{k}\{g_1(m_i(x_i,1))\}_{i=1}^{N}=\mathcal{F}_{k}\{m_i(-x_i,1)\}_{i=1}^{N}=\mathfrak{M}_k(-\mathfrak{X}_k,1)=g_1(\mathfrak{M}_k(\mathfrak{X}_k,1))$$ and $$\mathcal{F}_{k}\{g_2(m_i(x_i,1))\}_{i=1}^{N}=\mathcal{F}_{k}\{m_i(x_i+p,1)\}_{i=1}^{N}=\mathfrak{M}_k(\mathfrak{X}_k+p,1)=g_2(\mathfrak{M}_k(\mathfrak{X}_k,1))$$
 for $k\ne 0$.
 Let us check the only non-trivial equality $\mathcal{F}_{k}\{m_i(x_i+p,1)\}_{i=1}^{N}=\mathfrak{M}_k(\mathfrak{X}_k+p,1)$.

From the definition of $\mathfrak{M}_k$ and $\mathfrak{X}_k$, two relations $\mathfrak{M}_k^2=(\sum m_i\cosh(kx_i))^2-(\sum m_i\sinh(kx_i))^2$ and $\mathfrak{M}_k\sinh(k\mathfrak{X}_k)=\sum m_i\sinh(kx_i)$ are trivially derived, so additionally we obtain $\mathfrak{M}_k\cosh(k\mathfrak{X}_k)=\sum m_i\cosh(kx_i)$.

Hence we get
\begin{align*}
   &\mathcal{F}_{k}\{m_i(x_i+p,1)\}_{i=1}^{N}\\ =&\mathfrak{M}_k\left(\frac{1}{k}\sinh^{-1}\left(\frac{1}{\mathfrak{M}_k}{\sum m_i\sinh(k(x_i+p))}\right),1\right)\\
   =&\mathfrak{M}_k\left(\frac{1}{k}\sinh^{-1}\left(\frac{1}{\mathfrak{M}_k}\left(\cosh (kp)\sum m_i\sinh(kx_i)+\sinh(kp)\sum m_i\cosh(kx_i)\right)\right),1\right)\\
   =&\mathfrak{M}_k\left(\frac{1}{k}\sinh^{-1}\left(\cosh (kp)\sinh(k\mathfrak{X}_k)+\sinh(kp)\cosh(k\mathfrak{X}_k)\right),1\right)\\
   =&\mathfrak{M}_k\left(\frac{1}{k}\sinh^{-1}\sinh(k\mathfrak{X}_k+kp),1\right)\\
   =&\mathfrak{M}_k(\mathfrak{X}_k+p,1).
\end{align*}

For $N$ material vectors $\{m_i(\sinh x_i,\cosh x_i)\}_{i=1}^{N}$ in $\mathbb{H}^1$, we can define

\noindent $\mathcal{F}_{k}\{m_i(\sinh x_i,\cosh x_i)\}_{i=1}^{N}=\mathfrak{M}_k(\sinh \mathfrak{X}_k,\cosh \mathfrak{X}_k)$
for $k\in[0,\infty)$, using the same $\mathfrak{M}_k$ and $\mathfrak{X}_k$ as in the $\mathbb{E}^1$ case.
 The material vector system $\mathcal{F}_{k}$ in $\mathbb{H}^1$ also satisfies our six mass center axioms in a manner similar to the case in $\mathbb{E}^1$.

 Since $\mathcal{F}_{k_1}\ne \mathcal{F}_{k_2}$ for different $k_1, k_2\in[0,\infty)$, we conclude the proof.
\end{proof}
Note that $\mathcal{F}_{k}=\mathcal{F}_{-k}$. If we designate  a system where everything else is fixed but replaced by $\cos, \sin, \sin^{-1}$ instead of $\cosh, \sinh, \sinh^{-1}$ as the system $\mathcal{F}_{k}^s$ $(k>0)$, then the system $\mathcal{F}_{k}^s$ would satisfy the \emph{axiom of isometry invariance}, too. However, the \emph{axiom of single point} is not satisfied, since $\mathcal{F}_{k}^s\{m(\frac{\pi}{k},1)\}=m(0,1)$. This phenomenon on $\mathcal{F}_{k}^s$ corresponding to $\mathcal{F}_{k}$ essentially comes from the difference between $\sinh^{-1} \sinh x=x$ and $\sin^{-1} \sin x\ne x$. Therefore, we conjecture that there exists no other mass center system apart from $\mathcal{F}_{k}$ ($k\in [0,\infty$)) mentioned in $\mathbb{H}^1$ or $\mathbb{E}^1$.

\subsection{Mass Centers and centered masses of some  sets}
\subsubsection{Two Points and mass deviation}

In this subsection, we aim to discuss the location of the mass center of two material points and the associated centered mass. We also discuss the mass deviation between the centered mass and the total mass.

Let $\vc{a}$ and $\vc{b}$ be material vectors. Let 
\begin{equation}\label{s34}
    d([\vc{a}],[\vc{b}])=d, ~~d([\vc{a}],[\vc{a}\oplus\vc{b}])=d_1, ~~d([\vc{b}],[\vc{a}\oplus\vc{b}])=d_2. 
\end{equation}
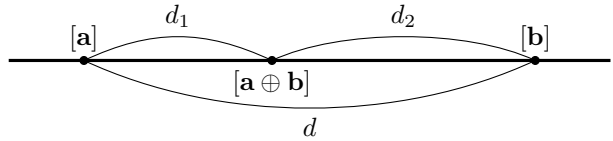
\begin{figure}[h!]
		\begin{center}\label{fig:2277}
   \input{line1.tex}
   \caption{two material points and their mass center}
		\end{center}
	\end{figure}

The next proposition describes the centered mass of two material points.
\begin{prop}\emph{(centered mass of two material points)}
 Then
\begin{align}
     {m_{\vc{a}\oplus\vc{b}}}
     &=\sqrt{{m_{\vc{a}}}^2+{m_{\vc{b}}}^2+2m_{\vc{a}}m_{\vc{b}}\cos_{X}d}\label{f37} \\
     &=m_{\vc{a}}\cos_{X}d_1 + m_{\vc{b}}\cos_{X}d_2.\label{f38}
\end{align}   
\end{prop}
\begin{proof}
In the non-Euclidean case, Equation (\ref{f37}) is derived from Equations (\ref{f3}) and (\ref{f9}) as follows:
\begin{align*}
    \delta_{x} {m_{\vc{a} \oplus \vc{b}}}^2 
    &= \langle \vc{a} + \vc{b}, \vc{a} + \vc{b} \rangle \\
    &= \langle \vc{a}, \vc{a} \rangle + \langle \vc{b}, \vc{b} \rangle + 2 \langle \vc{a}, \vc{b} \rangle \\
    &= \delta_{X} \left\{ m_{\vc{a}}^2 + m_{\vc{b}}^2 + 2 m_{\vc{a}} m_{\vc{b}} \cos_{X} d \right\}.
\end{align*}
Equation (\ref{f38}) is similarly derived from Equations (\ref{f3}) and (\ref{f9}) as follows:
\begin{align*}
    \delta_{X} {m_{\vc{a} \oplus \vc{b}}}^2
    &= \langle \vc{a} + \vc{b}, \vc{a} + \vc{b} \rangle \\
    &= \langle \vc{a}, \vc{a} + \vc{b} \rangle + \langle \vc{b}, \vc{a} + \vc{b} \rangle \\
    &= \delta_{X} \left\{ m_{\vc{a}} m_{\vc{a} \oplus \vc{b}} \cos_{X} d_1 + m_{\vc{b}} m_{\vc{a} \oplus \vc{b}} \cos_{X} d_2 \right\} \\
    &= \delta_{X} {m_{\vc{a} \oplus \vc{b}}} \left\{ m_{\vc{a}} \cos_{X} d_1 + m_{\vc{b}} \cos_{X} d_2 \right\}.
\end{align*}
In the Euclidean case, Equations (\ref{f37}) and (\ref{f38}) are trivial.
\end{proof}

According to the above proposition, the center of mass of two points depends not only on their masses but also on the distance between them.

The next proposition describes the location of the mass center of two material points. This proposition is called the \textit{lever law}.

\begin{prop}[Lever Law: location of mass center]\label{lever} Under the condition (\ref{s34}),
   we have
   \begin{align}
       & m_{\vc{a}} \sin_{X} d_1 = m_{\vc{b}} \sin_{X} d_2, \label{le1}      
   \end{align}
\end{prop}

\begin{proof}
    We omit the proof for Euclidean case and now give a proof for non-euclidean case.

    Squaring both sides of Equations (\ref{f37}) and (\ref{f38}), and then subtracting one from the other, we obtain Equation (\ref{le1}):
    \begin{align*}
        0 &= {m_{\vc{a}}}^2 + {m_{\vc{b}}}^2 +2 {m_{\vc{a}}} {m_{\vc{b}}}\cos_X (d_1+d_2)-\left(m_{\vc{a}}\cos_X d_1+m_{\vc{b}}\cos_X d_2\right)^2\\
        &= \delta_{X} \left\{ m_{\vc{a}} \sin_X d_1 - m_{\vc{b}} \sin_X d_2 \right\}^2.
    \end{align*}    
\end{proof}

The next proposition describes the deviation between the centered mass $m_{\text{cen}} \{\vc{a}, \vc{b}\}$ $= m_{\vc{a \oplus b}}$ and the total mass \( m_{\text{tot}} \{\vc{a}, \vc{b}\} = m_{\vc{a}} + m_{\vc{b}} \).

\begin{prop}[Deviation between the centered mass and the total mass]
    We have
    \[
    \left( m_{\text{cen}} \{\vc{a}, \vc{b}\} \right)^2 - \left( m_{\text{tot}} \{\vc{a}, \vc{b}\} \right)^2 = 2 m_{\vc{a}} m_{\vc{b}} \left( \cos_X d - 1 \right).
    \]
\end{prop}
\begin{proof}
    This proposition follows directly from Equation (\ref{f37}).
\end{proof}

According to the above proposition:\\
- In Euclidean space, the centered mass and the total mass are identical.\\
- In spherical space, the centered mass is smaller than the total mass, and the difference increases as the distance between the two points grows.\\
- In hyperbolic space, the centered mass is greater than the total mass, and the difference increases as the distance between the two points grows.

\section{Mass Centers of Manifolds }\label{section:mass centers of manifolds}
\subsection{Axioms and Uniqueness}

Let $M^k$ be a bounded $k$-dimensional manifold in $\mathbb{X}^n$ with a mass distribution $\rho$. The \textbf{total mass} of $M^k$ is defined as the integral of the mass distribution over $M^k$. It is denoted as
\[
m_{\text{tot}}(M^k, \rho)~\text{ or, simply, }~m_{\text{tot}}(M^k).
\]

By definition, a \textbf{mass center system for bounded k-manifolds} on $\mathbb{X}^{n}$ is a function
\[
\mathcal{F}\,:\,\text{bounded $k$-manifold with finite total mass}\mapsto\text{a material vector}
\]
which satisfies the \emph{axioms} we suggest below. By definition, we call the material vector $\mathcal{F}(M^k)$ the \textbf{mass center vector} of $M^k$, the mass of this mass center vector, the \textbf{centered mass} of $M^k$ and the point $\left[\mathcal{F}(M^k)\right]$, the \textbf{mass center} of $M^k$. Even though these 채concepts depend not only on $M^k$ but also on $\mathcal{F}$ at the moment, we denote the centered mass just as
\[
m_{\text{cen}}(M^k,\rho)~\text{ or, simply }~m_{\text{cen}}(M^k).
\]
and we will show the uniqueness of this mass center system $\mathcal{F}$.

In summery, for a $k$-submanifold $M^k$ with a mass distribution $\rho$ and with finite total mass,
\begin{equation}\label {manifold mass}
\mathcal{F}(M^k)=m_{\text{cen}}(M^k,\rho)\left[\mathcal{F}(M^k)\right]
\quad\text{and}\quad m_{\text{tot}}(M^k,\rho)=\int_{M^k}\rho\,\dd{V}_k.
\end{equation}
Now, we list the system of axioms for mass center of submanifolds.
\begin{itemize}
\item \emph{axiom of partition}: Let $M^k$ be a union of k-submanifolds. For any finite partition $M^k=\displaystyle\bigcup M_i$ into $k$-submanifolds $M_i$,
    \[
    \mathcal{F}(M^k)=\oplus_{i} \mathcal{F}(M_i).
    \]
\item \emph{axiom of convex hull}
: The mass center $\left[\mathcal{F}(M^k)\right]$ is located in the convex hull of $M^k$.
\item \emph{axiom of small object}
\footnote{\emph{Axiom of small object} together with \emph{axiom of convex hull} is the counterpart of \emph{axiom of single point}.}
: If the diameter of a submanifold is sufficiently small, then the centered mass and the total mass of it are sufficiently close, and the distance between this submanifold and the mass center of it are also sufficiently close, too. In more detail,
\begin{align*}
 \lim_{\text{diam}(M^k)\to0} \frac{m_{\text{cen}}(M^k,\rho)}{m_{\text{tot}}(M^k,\rho)}=1.
\end{align*}
\end{itemize}

The space $\mathbb{X}^n$ has its own metric $d$. Give the usual metric on the ambient space $\mathbb{R}^{n+1}$. Consider the metric $\tilde{d}$ on $\mathbb{X}^n$ induced by the metric on the ambient space:
\[
\tilde{d}(\vc{r}_1,\vc{r}_2)=|\vc{r}_1-\vc{r}_2|.
\]
Then the topologies induced by these two metrics are same.
The diameter of a set with respect to the metric $d$ is sufficiently small if and only if the diameter of this set with respect to $\tilde{d}$ is sufficiently small. In relation with the \emph{axiom of small object}, we may choose the notion of \emph{diameter} as
\[
\text{diam}(L)=\sup\{|\vc{v}-\vc{w}|\,|\,\vc{v},\vc{w}\in L\}
\quad\text{for } L\subset\mathbb{X}^n.
\]
We introduce a notation: For $\mathbb{X}\in\mathbb{X}^n$ and $\delta>0$,
\[
B_{\delta}(\mathbb{X}):=\{\vc{r}\in\mathbb{X}^{n}\,|\,|\vc{r}-\mathbb{X}|<\delta\} =\{\vc{v}\in\mathbb{R}^{n+1}\,|\,|\vc{v}-\mathbb{X}|<\delta\}\cap\mathbb{X}^n=:\mathbb{B}_{\delta}(\mathbb{X})\cap\mathbb{X}^n.
\]
It is obvious that $\text{diam}\left(B_{\delta}(\mathbb{X})\right)<2\delta$ for $\mathbb{X}\in\mathbb{X}^n\ne \mathbb{E}^n$, and $\text{diam}\left(B_{\delta}(\mathbb{X})\right)=2\delta$ for $\mathbb{X}\in \mathbb{E}^n$.

Note that for every $\mathbb{X}\in\mathbb{X}^n$ and for all sufficiently small $\delta$, the closure $\overline{B_{\delta}(\mathbb{X})}$ is convex\footnote{A subset $D\in\mathbb{X}^n$  is defined to be convex if any geodesic segment connecting any two points in $D$ lies within $D$. In particular, the set $B_{\delta}(\mathbb{X})$ is not convex,  if $\mathbb{X}^n=\mathbb{S}^n$ and $\delta>\sqrt{2}$.}.  With this property, we see that for some subset $L$ in $\overline{B_{\delta}(\mathbb{X})}$, the convex hull of $L$ is contained in $\overline{B_{\delta}(\mathbb{X})}$.\\

We present the uniqueness theorem now.
\begin{theorem}\label{manifold material vector} Let $M^k$ be a bounded $k$-submanifold in $\mathbb{X}^n$ on which mass distribution $\rho$ is defined and of which the total mass is finite. Then
  \[
  \mathcal{F}\left(M^k\right)=\int_{M^k}\rho\vc{r}\,\dd{V}_k,
  \]
where the vector $\vc{r}\in M^k\subset\mathbb{R}^{n+1}$ means the position vector.
\end{theorem}
\begin{proof}
Let $\epsilon>0$ be given.\\
Let $\overline{M^k}$ be the closure of $M^k$. The closure $\overline{M^k}$ is compact. By the \emph{axiom of small object}, for all $\mathbb{X}\in \overline{M^k}$, there exists $\delta=\delta(\mathbb{X})$ satisfying that $\delta(\mathbb{X})<\epsilon$ and
\begin{equation}\label{eqn:theorem2 0-th}
\left|\frac{m_{\text{cen}}(L)}{m_{\text{tot}}(L)} -1\right|<\epsilon ~\text{ whenever }~L\subset \overline{B_{\delta}(\mathbb{X})}.
\end{equation}
The collection $\left\{B_{\delta}(\mathbb{X})\,|\,\mathbb{X}\in\overline{M^k}\right\}$ is an open cover of $\overline{M^k}$. From this cover, we can choose a finite subcover $\{B_1,\dots,B_m\}$ of $\overline{M^k}$ and, consequently, of $M^k$. Consider the canonical partition of $M^k$ induced by this finite subcover: The \emph{part} in this canonical partition containing an element $\vc{r}$ is the intersection
\[
\left(\bigcap_{\vc{r}\in B_j}\overline{B_j}\right)\cap\left(\bigcap_{\vc{r}\notin B_k}(B_k)^c\right)\cap M^k.
\]
It is clear that this canonical partition is a finite one. There exists a finer partition $\{M_1,\dots,M_N\}$ of $M^k$ than this canonical partition such that, for some points $\vc{r}_i$ chosen from each $M_i$, we get
~$\displaystyle
\left|\sum_{i=1}^{N} \vc{r}_i\int_{M_i}\rho\,\dd{V}_k-\int_{M^k}\rho\vc{r}\,\dd{V}_k\right| <\epsilon
$ by Riemann integrable of $\int_{M^k}\rho\vc{r}\,\dd{V}_k$. In other notation, this is
\begin{equation}\label{eqn:theorem2 1-st}
{\left|\sum_{i=1}^{N}m_{\text{tot}}(M_i)\vc{r}_i-\int_{M^k}\rho\vc{r}\,\dd{V}_k\right| <\epsilon}.
\end{equation}
Each $M_i$ is contained in some $\overline{B_{\delta}(\mathbb{X})}$. By (\ref{eqn:theorem2 0-th}), we have
~$
\displaystyle\left|\frac{m_{\text{cen}}(M_i)}{m_{\text{tot}}(M_i)} -1\right|<\epsilon,
$\\
and, consequently,
\begin{equation}\label{eqn:theorem2 2-nd}
{\left|m_{\text{cen}}(M_i)-m_{\text{tot}}(M_i)\right|<\epsilon m_{\text{tot}}(M_i)}.
\end{equation}
Denote as $\mathcal{F}(M_i)=\vc{g}_i$. This is that $\vc{g}_i=m_{\text{cen}}(M_i)[\vc{g}_i]$.
\\
By the \emph{axiom of partition}, we have
\begin{equation}\label{eqn:theorem2 3-rd}
\displaystyle\mathcal{F}\left(M^k\right)=\oplus\vc{g}_i=\sum\vc{g}_i.
\end{equation}
By the \emph{axiom of convex hull}, the point $[\vc{g}_i]$ is contained in the convex hull of $M_i$. Since the convex hull of each $M_i$ is contained in  some $\overline{B_{\delta}(\mathbb{X})}$. In summery, we see
\[
\vc{r}_i\in M_i\subset\text{ convex hull of }M_i\subset\overline{B_{\delta}(\mathbb{X})}
\quad\text{and}\quad
[\vc{g}_i]\in\text{ convex hull of }M_i\subset\overline{B_{\delta}(\mathbb{X})}.
\]
Since the diameter of $\overline{B_{\delta}(\mathbb{X})}$ is less than or equal to $2\delta(\mathbb{X})$ which is less than $2\epsilon$, we see
\begin{equation}\label{eqn:theorem2 4-th}
{|[\vc{g}_i]-\vc{r}_i|<2\epsilon}.
\end{equation}
Then
\begin{align*}
  &\left|\vc{g}_i-m_{\text{tot}}(M_i)\vc{r}_i\right|
  \\
  &=\left|m_{\text{cen}}(M_i)[\vc{g}_i]-m_{\text{tot}}(M_i)\vc{r}_i\right|
  \\
  &\leq\left|m_{\text{cen}}(M_i)[\vc{g}_i]-m_{\text{tot}}(M_i)[\vc{g}_i]\right| +\left|m_{\text{tot}}(M_i)[\vc{g}_i]-m_{\text{tot}}(M_i)\vc{r}_i\right|
  \\
  &=\left|m_{\text{cen}}(M_i)-m_{\text{tot}}(M_i)\right|\,\left|[\vc{g}_i]\right| +m_{\text{tot}}(M_i)\left|[\vc{g}_i]-\vc{r}_i\right|
  \\
  &<\epsilon m_{\text{tot}}(M_i)\left|[\vc{g}_i]\right|+m_{\text{tot}}(M_i)2\epsilon
  \quad(\text{\small{by} (\ref{eqn:theorem2 2-nd}) and (\ref{eqn:theorem2 4-th})})
  \\
  &=\epsilon\left(\left|[\vc{g}_i]\right|+2\right)m_{\text{tot}}(M_i).
\end{align*}
Meanwhile, since each point $[\vc{g}_i]$ is in the convex hull of the bounded submainfold $M^k$ which is also bounded, there exists $K>0$ such that $|[\vc{g}_i]|<K$ for all $i$. Now, we can see that
\begin{equation}\label{eqn:theorem2 5-th}
\left|\vc{g}_i-m_{\text{tot}}(M_i)\vc{r}_i\right|
<\epsilon\,\left(K+2\right)m_{\text{tot}}(M_i).
\end{equation}
Then
\begin{align*}
  &\left|\mathcal{F}\left(M^k\right)-\int_{M^k}\rho\vc{r}\,\dd{V}_k\right|
  \\
  &=\left|\sum\vc{g}_i-\int_{M^k}\rho\vc{r}\,\dd{V}_k\right|
  \quad(\text{\small{by} (\ref{eqn:theorem2 3-rd})})
  \\
  &\le\left|\sum\vc{g}_i-\sum m_{\text{tot}}(M_i)\vc{r}_i\right| +\left|\sum m_{\text{tot}}(M_i)\vc{r}_i-\int_{M^k}\rho\vc{r}\,\dd{V}_k\right|
  \\
  &\le\sum\left|\vc{g}_i-m_{\text{tot}}(M_i)\vc{r}_i\right| +\left|\sum m_{\text{tot}}(M_i)\vc{r}_i-\int_{M^k}\rho\vc{r}\,\dd{V}_k\right|
  \\
  &<\sum \epsilon\left(K+2\right)m_{\text{tot}}(M_i)+\epsilon
  \quad(\text{\small{by} (\ref{eqn:theorem2 1-st}) and (\ref{eqn:theorem2 5-th})})
  \\
  &=\epsilon\left\{\left(K+2\right)m_{\text{tot}}(M^k)+1\right\}.
\end{align*}
Since $\epsilon$ is arbitrary and $\left\{\left(K+2\right)m_{\text{tot}}\left(M^k\right)+1\right\}$ is a constant, we can conclude that
\[
\mathcal{F}\left(M^k\right)=\int_{M^k}\rho\vc{r}\,\dd{V}_k.
\]
\end{proof}
\subsection{Mass Centers and centered masses of some manifolds}
We will call 1-dimensional centered mass as centered length, 2-dimensional centered mass as centered area, and higher dimensional centered mass as centered volume.  The following three sections provide direct examples of the calculation of the centered volume,

\subsubsection{Centered mass of a non-Euclidean $n$-dimensional ball}
For a spherical $n$-dimensional ball $B^n$ with radius $r~(0<r\le \pi)$, let us find the centered volume and mass center of the ball $B^n$.

By using the spherical coordinate $(\sin \varphi_1 \sin \varphi_2\cdots \sin \varphi_{n-1} \sin\varphi_n, \sin \varphi_1 \sin \varphi_2 $
$\cdots
\sin \varphi_{n-1} \cos\varphi_n, \ldots, \sin \varphi_1 \cos \varphi_2,  \cos \varphi_1)$, the ball $B^n$ lies on  $\mathbb{S}^n$ within $0\le\varphi_1\le r, 0\le\varphi_2,\ldots,\varphi_{n-1}\le \pi,  0\le\varphi_n\le 2\pi$. Condition $\rho=1$ and Theorem \ref{manifold material vector} shows that
\[  \mathcal{F}\left(B^n\right)=\int_{B^n}\vc{r}\,\dd{V}_n.
  \]
Also formula (\ref{manifold mass}) shows
$$\mathcal{F}(B^n)=m_{\text{cen}}(B^n)\left[\mathcal{F}(B^n)\right],$$
hence we know that
\begin{equation}\label{narea}
\mathcal{F}(B^n)=\left| \int_{B^n}  \vc{r}\, \dd V_n\right| \text{ and }  \left[\mathcal{F}(B^n)\right]=\left[\int_{B^n} \vc{r}\, \dd V_n\right]. \end{equation}
Since $\dd V_n=\sin^{n-1}\varphi_1 \sin^{n-2}\varphi_2\cdots \sin\varphi_{n-1} \dd \varphi_1\cdots \dd \varphi_n$, we have
$$\int_{B^n} \, \vc{r}~ \dd V_n =\int_0^{2\pi}\!\int_0^{\pi}\!\!\cdots\!\int_0^r (\sin \varphi_1 \sin \varphi_2\cdots \sin \varphi_{n-1} \sin\varphi_n, \ldots,\sin \varphi_1 \cos \varphi_2,  \cos \varphi_1)  \dd V_n,
$$
and we can calculate that
\begin{align*}
&\int_0^{2\pi}\!\!\int_0^{\pi}\!\cdots\!\int_0^r (\sin \varphi_1 \cdots \sin \varphi_n) \sin^{n-1} \varphi_1\cdots \sin \varphi_{n-1}\dd \varphi_1\cdots\dd \varphi_n\\
=&\int_0^r \sin^n \varphi_1 \dd \varphi_1\cdots \int_0^{\pi} \sin^2 \varphi_{n-1} \dd \varphi_{n-1} \int_0^{2\pi} \sin\varphi_{n} \dd \varphi_{n}\\=&0,\\
\vdots\\
&\int_0^{2\pi}\!\!\int_0^{\pi}\!\cdots\!\int_0^r (\sin \varphi_1 \cos \varphi_2) \sin^{n-1} \varphi_1\cdots \sin \varphi_{n-1}\dd \varphi_1\cdots\dd \varphi_n\\
=&\int_0^r \sin^n \varphi_1 \dd \varphi_1 \int_0^\pi \cos \varphi_2\sin^{n-2} \varphi_2 \dd \varphi_2 \cdots \int_0^{\pi} \sin \varphi_{n-1} \dd \varphi_{n-1} \int_0^{2\pi} 1\dd \varphi_{n}\\=&0,\\
\end{align*}
and
\begin{align*}
&\int_0^{2\pi}\!\!\int_0^{\pi}\!\cdots\!\int_0^r (\cos \varphi_1) \sin^{n-1} \varphi_1\cdots \sin \varphi_{n-1}\dd \varphi_1\cdots\dd \varphi_n\\
=&\int_0^r \cos \varphi_1 \sin^{n-1} \varphi_1 \dd \varphi_1 \int_0^{\pi} \sin^{n-2} \varphi_2\dd \varphi_2 \cdots \int_0^{\pi} \sin \varphi_{n-1} \dd \varphi_{n-1} \int_0^{2\pi} 1\dd \varphi_{n}\\
=&\left(\frac {1}{n} \sin^n r\right) \left( \sqrt{\pi}\frac{\Gamma(\frac{n-1}{2})}{\Gamma(\frac{n}{2})} \right) \cdots \left( \sqrt{\pi} \frac{\Gamma(\frac{3}{2})}{\Gamma(\frac{4}{2})} \right)\left( \sqrt{\pi} \frac{\Gamma(\frac{2}{2})}{\Gamma(\frac{3}{2})} \right) (2\pi)\\
=&\frac{\pi^{\frac{n}{2}}}{\Gamma(\frac{n}{2}+1)} \sin^n r.\\
\end{align*}

Hence the spherical ball with radius $r$ has centered volume of $\frac{\pi^{\frac{n}{2}}}{\Gamma(\frac{n}{2}+1)} \sin^n r$ and mass center of $(0,\ldots,0,1)$, which coincides to the original center of $B^n$.

Note that if $r=\pi$, then the ball $B^n$ has no mass center and if $\pi<r<2\pi$ and $n$ is positive odd integer, then the ball $B^n$ has minus mass center.

For a hyperbolic $n$-dimensional ball $B^n$, we can easily derive the centered volume by the below changing rule spherical to hyperbolic:
$$(\cos \varphi_1, \sin \varphi_1,\varphi_2,\ldots ,\varphi_n) \rightarrow
(\cosh \varphi_1, \sinh \varphi_1, \varphi_2,\ldots ,\varphi_n ).$$

We list the centered mass of unit density, i.e., $\rho=1$, of $k$-dimensional ball $B^k$ with a radius $r$ in Euclidean, spherical, and hyperbolic  spaces. We can see that the centered masses have simpler representations than total masses at the below $B^k(r)$ table.
$$
\def\arraystretch{1.3}
\begin{array}{c|c|c|c}
	\hline
	B^k(r) & \text{Euclidean} & \text{Spherical total mass} & \text{Hyperbolic total mass} \\ & \text{(centered) mass} & \text{and centered mass} & \text{and centered mass}\\
	\hline
k=1 & 2r & 2r, 2\sin r & 2r, 2\sinh r \\ 	\hline
k=2 & \pi r^2 & 4\pi \sin^2 \frac{r}{2}, \pi \sin^2 r & 4\pi \sinh^2 \frac{r}{2}, \pi \sinh^2 r\\
	\hline
k=3 & \frac{ 4\pi }{ 3}r^3 & 2\pi (r-\sin r\cos r),  & 2\pi (\sinh r\cosh r-r), \\
 &  &  \frac{ 4\pi }{ 3}\sin ^3 r &\frac{ 4\pi }{ 3}\sinh ^3 r \\
	\hline
	\vdots & \vdots & \vdots & \vdots  \\
	\hline
k=n & \frac{\pi^{\frac{n}{2}}}{\Gamma(\frac{n}{2}+1)}r^n &\frac{\pi^{\frac{n}{2}}}{\Gamma(\frac{n}{2}+1)}\int_0^r n\sin^{n-1} \varphi \dd \varphi , & \frac{\pi^{\frac{n}{2}}}{\Gamma(\frac{n}{2}+1)}\int_0^r n\sinh^{n-1} \varphi \dd \varphi ,\\
   & & \frac{\pi^{\frac{n}{2}}}{\Gamma(\frac{n}{2}+1)} \sin^n r & \frac{\pi^{\frac{n}{2}}}{\Gamma(\frac{n}{2}+1)} \sinh^n r \\
\end{array}$$

\subsubsection{Centered mass of a non-Euclidean $n$-dimensional sphere}
For an $(n-1)$-dimensional sphere $S^{n-1}$ with radius $r$, let us find the centered volume and mass center of the sphere $S^{n-1}$ on $\mathbb{X}^n$. Trivially, the mass center of $S^{n-1}$ is the same as the original center of $S^{n-1}$.  
Also we easily deduce that
$$m_{\text{tot}}(B^n(r))=\int_0^r m_{\text{tot}}(S^{n-1}(r)) \dd{r} \text{ and }m_{\text{cen}}(B^n(r))=\int_0^r m_{\text{cen}}(S^{n-1}(r)) \dd{r}.$$
Hence we obtain that
$$\frac{\dd}{\dd{r}}m_{\text{tot}}(B^n(r))=m_{\text{tot}}(S^{n-1}(r))
$$
and
$$
\frac{\dd}{\dd{r}}m_{\text{cen}}(B^n(r))=m_{\text{cen}}(S^{n-1}(r)).
$$

{\begin{equation*}
\def\arraystretch{1.3}
\begin{array}{c|c|c|c}
	\hline
	S^k(r) & \text{Euclidean} & \text{Spherical total mass} & \text{Hyperbolic total mass} \\ & \text{(centered) mass} & \text{and centered mass} & \text{and centered mass}\\
	\hline
k=0 & 2 & 2, 2\cos r & 2, 2\cosh r \\ 	\hline
k=1 & 2\pi r & 2\pi \sin r, 2\pi \sin r \cos r & 2\pi \sinh r, 2\pi \sinh r\cosh r\\
	\hline
k=2 & 4\pi r^2 & 4\pi \sin^2 r,  4\pi \sin ^2 r\cos r & 4\pi \sinh^2 r, 4\pi \sinh ^2 r\cosh r \\
	\hline
	\vdots & \vdots & \vdots & \vdots  \\
	\hline
k=n-1 & \frac{n\pi^{\frac{n}{2}}}{\Gamma(\frac{n}{2}+1)}r^{n-1} &\frac{n\pi^{\frac{n}{2}}}{\Gamma(\frac{n}{2}+1)}\sin^{n-1} r, & \frac{n\pi^{\frac{n}{2}}}{\Gamma(\frac{n}{2}+1)}\sinh^{n-1} r,\\
   & & \frac{n\pi^{\frac{n}{2}}}{\Gamma(\frac{n}{2}+1)} \sin^{n-1} r\cos r & \frac{n\pi^{\frac{n}{2}}}{\Gamma(\frac{n}{2}+1)} \sinh^{n-1} r\cosh r \\
\end{array}
\end{equation*}}

\subsubsection{Centered area of a non-Euclidean regular $n$-gonal disk}
For a spherical regular $n$-gonal $P_n^{\mathbb{S}}(a)$ with the same side edge length $a$, let us find the centered area and mass center of the regular $n$-gonal $P_n^{\mathbb{S}}(a)$.  The mass center is the same to the center of the $P_n^{\mathbb{S}}(a)$, trivially. First,
we consider a spherical right triangle with standard notation and $C=\frac{\pi}{2}$, then we get (see \cite{todhunter}, Art. 62)
\begin{equation}\label{right triangle2}
   \tan a =\tan A  \sin b, \text{ and } \tan b=\cos A \tan c.
\end{equation}

We can divide $P_n^{\mathbb{S}}(a)$ into $2n$ pieces of right triangles having a common origin of the $P_n^{\mathbb{S}}(a)$, we take one piece of right triangle $\triangle ABC$ with three sides $\frac{a}{2}, b, c$ and a right angle $C$, and a right sub-triangle $\triangle ADC$ with three sides $d, b, r$ (see Figure~\ref{regu2277}).

\begin{figure}[h!]
    \begin{center}
      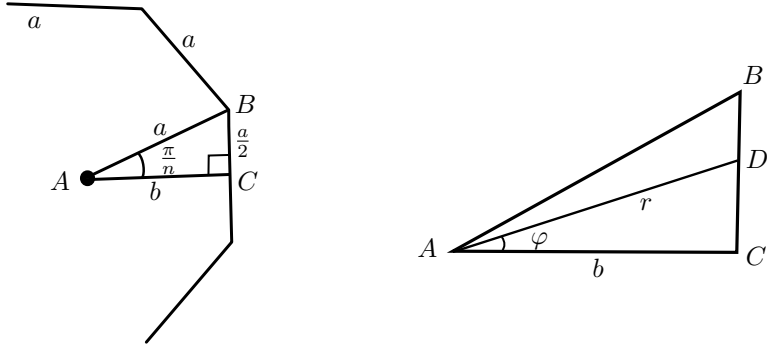
    
    \end{center}	
    \caption{a spherical or hyperbolic regular $n$-gon}\label{regu2277}
\end{figure}

We obtain the following formulas from (\ref{right triangle2}):
\begin{equation}\label{triangle2}
   \tan \frac{a}{2}=\tan \frac{\pi}{n}  \sin b, \text{ and } \tan b=\cos \varphi \tan r.
\end{equation}
Then we get \begin{equation}\label{triangle2s}\cot^2 b=\csc^2 b -1=\tan^2 \frac{\pi}{n} \cot^2 \frac{a}{2} -1,
\end{equation}
and
   \begin{align}\label{sin2r}
      \sin^2 r &=\frac{1}{1+\cot r^2}\nonumber\\
      &=\frac{1}{1+\cos^2 \varphi \cot^2 b}\nonumber\\
      &=\frac{1}{\sin^2 \varphi+\cos^2 \varphi \tan^2 \frac{\pi}{n} \cot^2 \frac{a}{2}},
   \end{align}
from the formula (\ref{triangle2}).
Let's adjust the three vertices $A, B, C$ of $\triangle ABC$ to the spherical coordinate system as $(\varphi_1, \varphi_2) = (0,0), (c, \frac{\pi}{n}), (b,0)$, respectively, with spherical coordinate defined as $(\sin \varphi_1 \sin \varphi_2, \sin \varphi_1 \cos \varphi_2, \cos \varphi_1).$ Then
the $P_n^{\mathbb{S}}(a)$ has mass center $(0,0,1)$, hence the $m_{\text{cen}}(P_n^{\mathbb{S}}(a))$ becomes the last coordinate integration \begin{equation}\label{intint}
    \int^{2\pi}_{0}\int^{r}_{0}~\sin \varphi_1\cos \varphi_1 \dd\varphi_1\dd\varphi_2.
\end{equation}
By the symmetry of the regular $n$-gonal and  \emph{axiom of partition}, the formula (\ref{intint}) is equal to
$$ 2n\int^{\frac{\pi}{n}}_{0}\int^{r}_{0}~\sin \varphi_1\cos \varphi_1 \dd\varphi_1\dd\varphi_2.$$
Finally, we can derive $m_{\text{cen}}(P_n^{\mathbb{S}}(a))=\frac{na}{2}\tan  \frac{a}{2} \cot \frac{\pi}{n}$  from the formula (\ref{sin2r}), and the below integration process:
\begin{align*}
n\int^{\frac{\pi}{n}}_{0}~\sin^2 r~ \dd\varphi_2 &=n\int^{\frac{\pi}{n}}_{0}~\frac{\dd \varphi_2}{\sin^2 \varphi_2+ \alpha^2 \cos^2 \varphi_2},\quad \alpha= \tan \frac{\pi}{n} \cot \frac{a}{2}\\
&=n\left[\frac{1}{\alpha}\tan^{-1}\left(\frac{1}{\alpha}\tan\varphi_2 \right)\right]^{\frac{\pi}{n}}_0\\
&=\frac{n}{\alpha} \frac{a}{2}\\
&= \frac{na}{2}\tan  \frac{a}{2} \cot \frac{\pi}{n}.
\end{align*}

For a hyperbolic regular $n$-gonal $P_n^{\mathbb{H}}(a)$ with the same side edge length $a$, let us find the centered area of that. We can see the hyperbolic counterpart of
(\ref{triangle2}), (\ref{triangle2s}), and (\ref{sin2r}):
\begin{align*}
   \tanh \frac{a}{2}=&\tan \frac{\pi}{n}  \sinh b, \text{ and } \tanh b=\cos \varphi \tanh r,\\
\coth^2 b=&\tan^2 \frac{\pi}{n} \coth^2 \frac{a}{2} +1,\\
      \sinh^2 r =&\frac{1}{\cos^2 \varphi \tan^2 \frac{\pi}{n} \coth^2 \frac{a}{2}-\sin^2 \varphi},
   \end{align*}
   and the corresponding integration process yields the expression $m_{\text{cen}}(P_n^{\mathbb{H}}(a)) = \frac{na}{2}\tanh \frac{a}{2} \cot \frac{\pi}{n}$:
\begin{align*}
n\int^{\frac{\pi}{n}}_{0}~\sinh^2 r~ \dd\varphi_2 &=n\int^{\frac{\pi}{n}}_{0}~\frac{\dd \varphi_2}{-\sin^2 \varphi_2+ \alpha^2 \cos^2 \varphi_2},\quad \alpha= \tan \frac{\pi}{n} \coth \frac{a}{2}\\
&=n\left[\frac{1}{2\alpha}\ln \left(\frac{\alpha+\tan \varphi_2}{\alpha-\tan \varphi_2}\right)\right]^{\frac{\pi}{n}}_0\\
&=\frac{n}{\alpha} \frac{a}{2}\\
&= \frac{na}{2}\tanh  \frac{a}{2} \cot \frac{\pi}{n}.
\end{align*}

\section{Pappus' Centroid Theorem and its application}

\subsection{Pappus solid and the statement of the theorem}
By a \emph{hyperplanar section}, we mean an $(n-1)$-dimensional manifold contained in a hyperplane.
\begin{definition}
A bounded $n$-dimensional manifold $\mathcal{P}$ in $\mathbb{X}^{n}$ with a mass distribution $\rho$ is called a \emph{Pappus solid along a centroid curve} $C$ or simply a \emph{Pappus solid}  if the solid $\mathcal{P}$ is laminated by the hyperplanar sections with the restricted mass distribution $\rho$ satisfying the conditions explained below.\\
- The mass centers of all concerned hyperplanar sections are mutually disjoint and they form the curve $C$.\\
- The arclength parametrization of the curve $C$ is $\mathcal{C}^1$-differentiable.\\
- The hyperplanar section deforms smoothly along the centroid curve and, hence, the hyperplane containing a hyperplanar section also does.
\end{definition}
The centroid curve of a Pappus solid is not necessarily contained in the solid. The hyperplane which contains a hyperplanar section is not necessarily perpendicular to the centroid curve. We also simply call a hyperplanar section just a \emph{section}. 
We describe a Pappus solid in detail:
\begin{itemize}
\item In the above definition, let the centroid curve $C$ is parametrized by the arclength as $\vc{c}(t),~(t\in I\subset\mathbb{R})$. The velocity vector of this curve is a unit tangent vector. We denote this vector as $\vc{t}(t)$ :
    $
    ~~\vc{t}(t)=\vc{c}'(t), ~ \langle\vc{t},\vc{t}\rangle=1
    $
\item The section mass-centered at $\vc{c}(t)$  is denoted as $L_t$.
\item Let $\mathbb{X}^{n-1}_{\vc{p}(t),\vc{c}(t)}$ be the hyperplane containing the section $L_t$.
\item The polar vector $\vc{p}(t)$ of this hyperplane is selected so that, for all $t\in I$,
    \[
    \langle\vc{t}(t),\vc{p}(t)\rangle\geq0.
    \]
\item Let $\theta_t$ be the slant angle of this hyperplane with respect to the centroid curve. By Equation (\ref{f10}), we see
    \[
    \cos\theta_t=\langle\vc{t}(t),\vc{p}(t)\rangle.
    \]
\item The Pappus solid $\mathcal{P}$ described above is denoted as
    \[
    \mathcal{P}=\mathcal{P}\big[\vc{c}(t),L_t,\theta_t\big].
    \]
\end{itemize}

\begin{theorem}\label{pappus}
  Let $\mathcal{P}=\mathcal{P}\big[\vc{c}(t),L_t,\theta_t\big]$ be a Pappus solid with a mass distribution $\rho$. Then the total mass of this solid is the total mass of the centroid curve where the mass distribution $\tilde{\rho}(t)$ at the point $\vc{c}(t)$ is given by $\tilde{\rho}(t)=m_{\text{cen}}(L_t,\rho)\cos\theta_t$ \emph{:}
    \begin{equation}\label{pappus' centroid theorem}
    m_{\text{tot}}\Big(\mathcal{P}\big[\vc{c}(t),L_t,\theta_t)\big],\rho\Big)
    =m_{\text{tot}}\Big(\vc{c}(t),
    \tilde{\rho}=m_{\text{cen}}(L_t,\rho)\cos\theta_t)\Big).
  \end{equation}
\end{theorem}
Section \ref{section:Pappus' Centroid Theorem on Volume} and \ref{Pappus Euclid thm} consist of the proof of this theorem.

\subsection{Non-Euclidean Pappus' Centroid Theorem on Volume}\label{section:Pappus' Centroid Theorem on Volume}
In this section, we will present the necessary tools and proof for the non-Euclidean version of Pappus' centroid theorem.

\subsubsection{Moving basis and parametrization of Pappus Solid}
The Pappus solid $\mathcal{P}=\mathcal{P}\big[\vc{c}(t),L_t,\theta_t\big]$ immediately determines the velocity vector $\vc{t}(t)$ and the polar vector $\vc{p}(t)$. Following the discussion in Section, we can construct an orthonormal basis $\{\vc{n}_1(t),\dots,\vc{n}_{n-1}(t)\}$ for the space $\text{Span}\{\vc{c}(t),\vc{p}(t)\}^{\perp}$ $\subset\mathbb{R}^{n+1}$. Then the set of vectors
\[
\{\vc{p}(t),\vc{n}_1(t),\dots,\vc{n}_{n-1}(t),\vc{c}(t)\}
\]
is orthonormal with respect to the $2$-form $\langle~,~\rangle$ for all $t\in I$. This is an \emph{moving basis} associated to the point $\vc{c}(t)$. This moving basis is not a moving frame because the vector $\vc{c}$ is not a tangent vector and the vector $\vc{p}$ is not necessarily a tangent vector if the space is Euclidean. Others are tangential. This moving basis is assumably $\mathcal{C}^1$, even though we do not provide a verification of its existence.  For each $t\in I$, define the square matrix $F_t$ as
\begin{equation}\label{eqn:moving frame matrix F_t}
  F_t=[\vc{p}(t),\vc{n}_1(t),\cdots,\vc{n}_{n-1}(t),\vc{c}(t)].
\end{equation}

Consider Equation (\ref{eqn:X^{n-1}(p,c)}). It is clear that a point $\vc{r}$ in $\mathbb{X}^{n-1}_{\vc{p}(t),\vc{c}(t)}$ can be parametrized as
\begin{equation}\label{rft}
  \vc{r}=\vc{r}(t,\varphi_1,\dots,\varphi_{n-1})=F_t\begin{bmatrix}
                                             0 \\ \Phi^{n-1}(\varphi_1,\dots,\varphi_{n-1})
                                          \end{bmatrix}.
\end{equation}
The section $L_t$ is contained in $\mathbb{X}^{n-1}_{\vc{p}(t),\vc{c}(t)}$. The domain of definition for the Pappus solid $\mathcal{P}$ is described like this : For each $t\in I$, the $(n-1)$-tuple $(\varphi_1,\dots,\varphi_{n-1})$ varies in some region $D_t\subset\mathbb{R}^{n-1}$ depending on $t$ :
\begin{align*}\label{}
  &\mathcal{P}\big[\vc{c}(t),L_t,\theta_t\big]
  \\
  =&
  \left\{\left.\vc{r}=F_t\begin{bmatrix}
                     0 \\ \Phi^{n-1}(\varphi_1,\dots,\varphi_{n-1})
                     \end{bmatrix}
  \,\right|\,t\in I,~ (\varphi_1,\dots,\varphi_{n-1})\in D_t\subset\mathbb{R}^{n-1}
  \right\}.
\end{align*}
This is a parametrization of a Pappus solid. Observe that this parametrization is regular.

For notational simplicity, we denote $\Phi^{n-1}$ simply as $\Phi$ and denote the Jacobian matrix
$\left(\Phi_{\varphi_1}~\Phi_{\varphi_2}~\cdots~ \Phi_{\varphi_{n-1}}\right)$ of $\Phi$ as $D\Phi$ :
\[
\Phi(\varphi_1,\dots,\varphi_{n-1})=(f_1,\dots,f_{n})^T,
\quad D\Phi=\left(\Phi_{\varphi_1}~\Phi_{\varphi_2}~\cdots~ \Phi_{\varphi_{n-1}}\right).
\]

\begin{lemma}\label{lem4m}
Let $\mathcal{P}=\mathcal{P}\big[\vc{c}(t),L_t,\theta_t\big]$ be a Pappus solid with a mass distribution $\rho$. A parametrization is given by Equation (\ref{rft}) with the parameter domain $D_t$. For each leaf $L_t~(t\in I)$, we have
\begin{equation}\label{f57}
\int_{D_t}\rho f_i\sqrt{\left|\det\left\{(D\Phi)^TJ_nD\Phi\right\}\right|\,}\,\, \dd{\varphi_1}\cdots\dd{\varphi_{n-1}}
=\begin{cases}
   0, & \emph{\mbox{if }}~ i<n \\
   m_{\text{cen}}(L_t,\rho), & \emph{\mbox{if }}~i=n.
 \end{cases}
\end{equation}
\end{lemma}
\begin{proof}
The volume element $\dd{V}_{n-1}$ of the section $L_t$ is given by
\[
\dd{V}_{n-1} =\sqrt{\left|\det\left\{(D\Phi)^TJ_nD\Phi\right\}\right|\,}\,\, \dd{\varphi_1}\cdots\dd{\varphi_{n-1}}.
\]
The mass center of the section $L_t$ is $\vc{c}(t)$. Hence the mass center vector is
\begin{align*}
&~~m_{\text{cen}}(L_t,\rho)\vc{c}(t)
\\
&=\int_{L_t}\rho\,\vc{r}\,\dd{V}_{n-1}
\\
&=\int_{D_t}\rho (f_1\vc{n}_1+\dots+f_{n-1}\vc{n}_{n-1}+f_{n}\vc{c})\,
\sqrt{\left|\det\left\{(D\Phi)^TJ_nD\Phi\right\}\right|\,}\,\, \dd{\varphi_1}\cdots\dd{\varphi_{n-1}}
\\
&=\int_{D_t}\rho \left(\sum_{i=1}^{n-1}f_i\vc{n}_i+f_{n}\vc{c}\right)\,
\sqrt{\left|\det\left\{(D\Phi)^TJ_nD\Phi\right\}\right|\,}\,\, \dd{\varphi_1}\cdots\dd{\varphi_{n-1}}
\\
&=\sum_{i=1}^{n-1}\left(\int_{D_t}\rho f_i
\sqrt{\left|\det\left\{(D\Phi)^TJ_nD\Phi\right\}\right|\,}\,\, \dd{\varphi_1}\cdots\dd{\varphi_{n-1}}\right)\vc{n}_i(t)
\\
&\quad+\left(\int_{D_t}\rho f_{n}
\sqrt{\left|\det\left\{(D\Phi)^TJ_nD\Phi\right\}\right|\,}\,\, \dd{\varphi_1}\cdots\dd{\varphi_{n-1}}\right)\vc{c}(t).
\end{align*}
The set $\{\vc{n}_1(t),\dots,\vc{n}_{n-1}(t),\vc{c}(t)\}$ is independent. This completes the proof.
\end{proof}

\subsubsection{Derivative of the basis matrix}
In the next lemma, we investigate the relation between the basis matrix $F_t$ given in (\ref{eqn:moving frame matrix F_t}) and its derivative ${F_t}'=\displaystyle\frac{\dd{F_t}}{\dd{t}}$.

\begin{lemma}
The basis matrix $F_t$ given in \emph{(\ref{eqn:moving frame matrix F_t})} satisfies the system of differential equations
\begin{equation}\label{eqn:ODEs for frame}
{F_t}'=F_t J_{n+1} \Lambda_t
\end{equation}
where $\Lambda_t$ is of the form
\[
\Lambda_t=
\left(
\begin{array}{c|c}
  0& \begin{array}{c c c c}
     \alpha_1(t) & \cdots & \alpha_{n-1}(t) & {\cos\theta_t}
     \end{array} \\
  \hline
  \begin{array}{c}
  -\alpha_1(t) \\ \vdots \\ -\alpha_{n-1}(t) \\ {-\cos\theta_t}
  \end{array} & \lambda_t
\end{array}
\right)
\]
for some skew-symmetric matrix $\lambda_t$. Every entry functions of $\Lambda_t$ is a continuous function depending only on $t$.
\end{lemma}
\begin{proof}
The matrix $F_t$ is orthogonal:
\[
\left(F_t\right)^TJ_{n+1}F_t=J_{n+1}.
\]
Differentiating both sides of the above equation, we get
\[
O=\left({F_t}^{T}\right)'J_{n+1}F_t+{F_t}^{T}J_{n+1}{F_t}'
=\left({F_t}^{T}J_{n+1}{F_t}'\right)^{T}+{F_t}^{T}J_{n+1}{F_t}'.
\]
This means that the matrix $\Lambda_t:={F_t}^{T}J_{n+1}{F_t}'$ is skew-symmetric.\\
Since ${F_t}^{T}J_{n+1}F_tJ_{n+1}=J_{n+1}J_{n+1}=I_{n+1}$, we have $\left({F_t}^{T}J_{n+1}\right)^{-1}=F_tJ_{n+1}$  and
\begin{equation*}\label{}
  {F_t}'=F_tJ_{n+1}\Lambda_t.
\end{equation*}
Let's look into this skew-symmetric matrix $\Lambda_t$ in more detail. Since
\begin{equation*}
{F_t}'=[\vc{p}',{\vc{n}_1}',\cdots,{\vc{n}_{n-1}}',{\vc{c}'}]
=[\vc{p}',{\vc{n}_1}',\cdots,{\vc{n}_{n-1}}',{\vc{t}}]
\end{equation*}
and
\begin{equation*}
{F_t}'=F_tJ_{n+1}\Lambda_t
=[\vc{p},{\vc{n}_1},\cdots,{\vc{n}_{n-1}},\delta_{X}\vc{c}]\Lambda_t,
\end{equation*}
we have that
\begin{equation}\label{eqn:F_t'}
  [\vc{p}',{\vc{n}_1}',\cdots,{\vc{n}_{n-1}}',\vc{t}]=
  [\vc{p},{\vc{n}_1},\cdots,{\vc{n}_{n-1}},\delta_{X}\vc{c}]\Lambda_t.
\end{equation}
Let the last column of the matrix $\Lambda_t$ be the vector $[d_0,d_1,\dots,d_{n-1},d_{n}]^T$.\\
From Equation (\ref{eqn:F_t'}), we can see that
\[
\vc{t}(t) =d_0\vc{p}(t)+d_1\vc{n}_1(t)+\dots+d_{n-1}\vc{n}_{n-1}(t)+\delta_{X}d_n\vc{c}(t).
\]
Apply $\langle~,\vc{p}\rangle$ on both sides of the above equation. By the orthonormality of the basis $\{\vc{p}(t),\vc{n}_1(t),\dots,\vc{n}_{n-1}(t),\vc{c}(t)\}$ , we can see that
\begin{equation*}
  d_0 = \langle\vc{t}(t),\vc{p}(t)\rangle=\cos\theta_t.
\end{equation*}
and , hence, the skew-symmetric matrix $\Lambda_t$ is of the form
\[
\Lambda_t=
\left(
\begin{array}{c|c}
  0& \begin{array}{c c c c}
     \alpha_1(t) & \cdots & \alpha_{n-1}(t) & \cos\theta_t
     \end{array} \\
  \hline
  \begin{array}{c}
  -\alpha_1(t) \\ \vdots \\ -\alpha_{n-1}(t) \\ -\cos\theta_t
  \end{array} & \lambda_t
\end{array}
\right)
\]
for some skew-symmetric matrix $\lambda_t$. This completes proof.
\end{proof}

\subsubsection{Proof of the Pappus theorem}
Let $D\vc{r}$ is the Jacobian matrix of $\vc{r}$.
Then we can see that
\begin{align*}
D\vc{r}&=
\begin{pmatrix}
\vc{r}_{t}&\vc{r}_{\varphi_1}&\cdots&\vc{r}_{\varphi_{n-1}}
\end{pmatrix}
\\
&=
\left(
\begin{array}{c c c c}
  {F_t}'\begin{bmatrix}
                0 \\
                \Phi
              \end{bmatrix} & F_t\begin{bmatrix}
                                0 \\
                                \Phi_{\varphi_1}
                              \end{bmatrix} & \cdots & F_t\begin{bmatrix}
                                                         0 \\
                                                         \Phi_{\varphi_{n-1}}
                                                       \end{bmatrix}
\end{array}
\right)
\\
&=
\left(
\begin{array}{c c c c}
  F_t J_{n+1}\Lambda_t\begin{bmatrix}
                0 \\
                \Phi
              \end{bmatrix} & F_t\begin{bmatrix}
                                0 \\
                                \Phi_{\varphi_1}
                              \end{bmatrix} & \cdots & F_t\begin{bmatrix}
                                                         0 \\
                                                         \Phi_{\varphi_{n-1}}
                                                       \end{bmatrix}
\end{array}
\right)\quad (\text{\small by Equation (\ref{eqn:ODEs for frame})})
\\
&=F_t\left(
\begin{array}{c c c c}
  J_{n+1}\Lambda_t\begin{bmatrix}
                0 \\
                \Phi
              \end{bmatrix} & \begin{bmatrix}
                                0 \\
                                \Phi_{\varphi_1}
                              \end{bmatrix} & \cdots & \begin{bmatrix}
                                                         0 \\
                                                         \Phi_{\varphi_{n-1}}
                                                       \end{bmatrix}
\end{array}
\right)
\\
&=
F_t\left(
\begin{array}{c|c c c}
  \alpha_1f_1+\dots+\alpha_{n-1}f_{n-1}+\cos\theta_t f_{n} & 0 & \cdots & 0 \\
  \hline
  J_{n}\lambda_t\Phi & \Phi_{\varphi_1} & \cdots & \Phi_{\varphi_{n-1}}
\end{array}
\right).
\end{align*}

\begin{lemma}
\[
J_{n}\lambda_t\Phi\in T_{\Phi}\mathbb{X}^n=\text{\emph{Span}}\{\Phi_{\varphi_1},\dots,\Phi_{\varphi_{n-1}}\},\quad
J_{n}\lambda_t\Phi=a_1\Phi_{\varphi_1}+\dots+a_{n-1}\Phi_{\varphi_{n-1}}.
\]
for some real numbers $a_1,\dots,a_{n-1}$ depending on $t$.
\end{lemma}
\begin{proof}
  Since $\langle\Phi,\Phi\rangle=1$ or $-1$, we have
  \[
  \langle\Phi,\Phi_{\varphi_i}\rangle=0~\text{ for all }~i=1,2,\dots,n-1.
  \]
  Since the parametrization $\Phi(\varphi_1,\dots,\varphi_{n-1})$ is regular, the set $\{\Phi_{\varphi_1},\dots,\Phi_{\varphi_{n-1}}\}$ is a basis for the orthogonal complement $\Phi(\varphi_1,\dots,\varphi_{n-1})^{\perp}$ for all $(\varphi_1,\dots,\varphi_{n-1})$:
  \begin{equation}\label{eqn:Phi^{perp}}
  \Phi(\varphi_1,\dots,\varphi_{n-1})^{\perp}
  =\text{Span}\{\Phi_{\varphi_1},\dots,\Phi_{\varphi_{n-1}}\}.
  \end{equation}
  We claim that
  \begin{equation}\label{eqn:<Phi,J_{n}lambda(t)Phi>=0}
  \langle J_{n}\lambda_t\Phi,\Phi \rangle=0.
  \end{equation}
  To see this, observe that
  \[
  \langle J_{n}\lambda_t\Phi,\Phi \rangle=\Phi^TJ_{n}J_{n}\lambda_t\Phi =\Phi^T\lambda_t\Phi,
  \]
  and symmetric property of $\langle~ , ~\rangle$ shows that
  \[
  \Phi^T\lambda_t\Phi=\left(\Phi^T\lambda_t\Phi\right)^T=\Phi^T(\lambda_t)^T\Phi =-\Phi^T\lambda_t\Phi.
  \]
  Hence, by Equation (\ref{eqn:Phi^{perp}}) and Equation (\ref{eqn:<Phi,J_{n}lambda(t)Phi>=0}),
  \[
  J_{n}\lambda_t\Phi\in\text{Span}\{\Phi_{\varphi_1},\dots,\Phi_{\varphi_{n-1}}\}.
  \]
\end{proof}

By the above lemma, we can make the Jacobian matrix $D\vc{r}$ simpler via some elementary column operations:
\begin{eqnarray*}
D\vc{r}&=&F_t\left(
\begin{array}{c|c c c}
  \alpha_1f_1+\dots+\alpha_{n-1}f_{n-1}+\cos\theta_t f_{n} & 0 & \cdots & 0 \\
  \hline
  J_n\lambda_t\Phi & \Phi_{\varphi_1} & \cdots & \Phi_{\varphi_{n-1}}
\end{array}
\right)
\\
&=&F_t\left(
\begin{array}{c|c c c}
  \alpha_1f_1+\dots+\alpha_{n-1}f_{n-1}+\cos\theta_t f_{n} & 0 & \cdots & 0 \\
  \hline
  a_1\Phi_{\varphi_1}+\dots+a_{n-1}\Phi_{\varphi_{n-1}} & \Phi_{\varphi_1} & \cdots & \Phi_{\varphi_{n-1}}
\end{array}
\right)
\\
& &\quad\qquad \downarrow{\text{ \small{elementary column operations}}}
\\
& &F_t\left(
\begin{array}{c|c c c}
  \alpha_1f_1+\dots+\alpha_{n-1}f_{n-1}+\cos\theta_t f_{n} & 0 & \cdots & 0 \\
  \hline
  \vcz & \Phi_{\varphi_1} & \cdots & \Phi_{\varphi_{n-1}}
\end{array}
\right).
\\
&=&F_t\left(
\begin{array}{c|c}
  \alpha_1f_1+\dots+\alpha_{n-1}f_{n-1}+\cos\theta_t f_{n} & \vcz \\
  \hline
  \vcz & D\Phi
\end{array}
\right)
\end{eqnarray*}
This means that
\[
D\vc{r}=F_t\left(
\begin{array}{c|c}
  \alpha_1f_1+\dots+\alpha_{n-1}f_{n-1}+\cos\theta_t f_{n} & \vcz\\
  \hline
  \vcz & D\Phi
\end{array}
\right)E
\]
where the matrix $E$ is a product of some elementary matrices with determinant $1$. To find the volume element, we need to calculate $\sqrt{\left|\det\left\{(D\vc{r})^TJ_{n+1}D\vc{r}\right\}\right|~}$. It is easy to see that
\begin{eqnarray*}
& &\left|\det\left\{(D\vc{r})^TJ_{n+1}D\vc{r}\right\}\right|
\\
&=&
\left|
\det\left(
\begin{array}{c|c c c}
  (\alpha_1f_1+\dots+\alpha_{n-1}f_{n-1}+\cos\theta_t f_{n})^2 & \vcz \\
  \hline
  \vcz & (D\Phi)^TJ_{n}D\Phi
\end{array}
\right)
\right|
\\
&=&(\alpha_1f_1+\dots+\alpha_{n-1}f_{n-1}+\cos\theta_t f_{n})^2 \left|\det\left\{(D\Phi)^TJ_{n}D\Phi\right\}\right|.
\end{eqnarray*}
Hence, the volume element $\dd{V}_n$ is
\begin{eqnarray*}
   & & \dd{V}_n \\
   &=& \sqrt{\left|\det\left\{(D\vc{r})^TJ_{n+1}D\vc{r}\right\}\right|\,}\, \dd{\varphi_1}\cdots\dd{\varphi_{n-1}}\dd{t} \\
   &=& \left|\sum_{i=1}^{n-1}\alpha_if_i+\cos\theta_t f_{n}\right|
   \sqrt{\left|\det\left\{(D\Phi)^TJ_{n}D\Phi\right\}\right|\,}\, \dd{\varphi_1}\cdots\dd{\varphi_{n-1}}\dd{t}.
\end{eqnarray*}
Since the parametrization of $\vc{r}$ is regular, the volume form never vanishes. Hence, for $\epsilon_0=1$ or $-1$, we see
\[
\left|\sum_{i=1}^{n-1}\alpha_if_i+\cos\theta_t f_{n}\right|
=\epsilon_0\left(\sum_{i=1}^{n-1}\alpha_if_i+\cos\theta_t f_{n}\right)
\]

Now we can calculate the total mass :
\begin{eqnarray*}
& &m_{\text{tot}}\Big(\mathcal{P},\rho\Big)
\\
&=&\int_{\mathcal{P}}\rho\,\dd{V}_n
\\
&=&\int_{I}\int_{D_t} \rho\,\epsilon_0 \left(\sum_{i=1}^{n-1}\alpha_if_i+\cos\theta_t f_{n}\right)
\sqrt{\left|\det\left\{(D\Phi)^TJ_{n}D\Phi\right\}\right|\,}\, \dd{\varphi_1}\cdots\dd{\varphi_{n-1}}\dd{t}
\\
&=&\epsilon_0\sum_{i=1}^{n-1}\int_{I}\alpha_i\,\left(\int_{D_t}\rho f_i
\sqrt{\left|\det\left\{(D\Phi)^TJ_{n}D\Phi\right\}\right|\,}\, \dd{\varphi_1}\cdots\dd{\varphi_{n-1}}\right)\,\dd{t}
\\
& &+\epsilon_0\int_{I}\cos\theta_t\left(\int_{D_t}\rho f_{n}
\sqrt{\left|\det\left\{(D\Phi)^TJ_{n}D\Phi\right\}\right|\,}\, \dd{\varphi_1}\cdots\dd{\varphi_{n-1}}\right)\,\dd{t}
\\
&=&\epsilon_0\int_{I}m_{\text{cen}}(L_t,\rho)\cos\theta_t\,\dd t
\quad(\text{\small by Equation (\ref{f57}) in Lemma \ref{lem4m}})
\\
&=&\int_{\vc{c}(t)}m_{\text{cen}}(L_t,\rho)\cos\theta_t\,\dd t.
\quad ({\text{\small line integral interpretation and $\epsilon_0$ must be
 $1$}})
\\
&=&m_{\text{tot}}\Big(\vc{c}(t),\tilde{\rho}=m_{\text{cen}}(L_t,\rho) \cos\theta_t\Big).
\end{eqnarray*}
This completes the proof of Pappus' theorem on volume.


\subsection{Euclidean Pappus' Centroid Theorem on Volumes}\label{Pappus Euclid thm}
In this section, we will present the necessary tools and proof for the Euclidean version of Pappus' centroid theorem.

\subsubsection{Moving basis and parametrization of Pappus Solid}
The Pappus solid $\mathcal{P}=\mathcal{P}\big[\vc{c}(t),L_t,\theta_t\big]$ immediately determines the velocity vector $\vc{t}(t)$ and the unit normal vector $\vc{n}(t)\in T_{\vc{c}(t)}\mathbb{E}^n$ to the hyperplane containing the section $L_t$. Following the discussion in Section, we can construct an orthonormal basis $\{\vc{n}_1(t),\dots,\vc{n}_{n-1}(t)\}$ for the space $\text{Span}\{\mathbb{E}_{n+1},\vc{n}(t)\}^{\perp}$ $\subset\mathbb{R}^{n+1}$.

 Then the set $\{\vc{n},\vc{n}_1,\dots,\vc{n}_{n-1}\}$ is a orthonormal basis for $T_{\vc{c}}\mathbb{E}^n$ with respect to the dot product. Since each vector in $T_{\vc{c}}\mathbb{E}^n$ is orthogonal to $\mathbb{E}_{n+1}\in\mathbb{R}^{n+1}$, the set
\[
\{\vc{n},\vc{n}_1,\dots,\vc{n}_{n-1},\mathbb{E}_{n+1}\}
\]
is orthonormal with respect to the dot-product.
Let
\[
\mathbb{E}^{n-1}_{\vc{n},\vc{c}}
\]
be the hyperplane of $\mathbb{E}^n$ that passes through $\vc{c}$ and has $\vc{n}$ as a normal vector. Then a point $\vc{r}$ in $\mathbb{E}^{n-1}_{\vc{n},\vc{c}}$ is of the form
\begin{equation}\label{eqn:E^{n-1}(t,r)}
  \vc{r}=\vc{c}+x_1\vc{n}_1+\dots+x_{n-1}\vc{n}_{n-1}
  =\vc{c}+[\vc{n},\vc{n}_1,\cdots,\vc{n}_{n-1},\mathbb{E}_{n+1}]
  \begin{bmatrix}
    0\\ x_1 \\ \vdots \\ x_{n-1} \\ 0
  \end{bmatrix}.
\end{equation}

Let $C$ be a $\mathcal{C}^1$ curve on $\mathbb{E}^n$ parametrized by the arclength as
\[
\vc{c}(t),\quad t\in I\subset\mathbb{R}.
\]
Then it is clear that $\vc{t}(t)\in T_{\vc{c}(t)}\mathbb{E}^n$ and $\vc{t}(t)\cdot\vc{t}(t)=1$ for all $t\in I$.
There exists a set of $\mathcal{C}^1$ parametrized vectors $\{\vc{n}_1(t),\dots,\vc{n}_{n-1}(t)\}$ such that
\[
\{\vc{n}(t),\vc{n}_1(t),\dots,\vc{n}_{n-1}(t),\mathbb{E}_{n+1}\}
\]
is orthonormal with respect to the dot-product for all $t\in I$. For each $t\in I$, define the square matrix $F_t$ as
\begin{equation}\label{eqn:euclidean matrix F_t}
  F_t=[\vc{n}(t),\vc{n}_1(t),\cdots,\vc{n}_{n-1}(t),\mathbb{E}_{n+1}].
\end{equation}

From the expression (\ref{eqn:E^{n-1}(t,r)}), it is clear that a point $\vc{r}$ in $\mathbb{E}^{n-1}_{\vc{n}(t),\vc{c}(t)}$ can be parametrized as
\begin{equation*}
  \vc{r}=\vc{r}(t,x_1,\dots,x_{n-1})=\vc{c}(t)+F_t\begin{bmatrix}
                                                   0 \\
                                                   x_1 \\
                                                   \vdots \\
                                                   x_{n-1} \\
                                                   0
                                                 \end{bmatrix}.
\end{equation*}
The section $L_t$ is a subset of the hyperplane $\mathbb{E}^{n-1}_{\vc{n}(t),\vc{c}(t)}$, and the normal vector $\vc{n}(t)$ of the hyperplane intersects the curve $C$ with a slant angle $\theta_t$ at $\vc{c}(t)$.
This is nothing other than a parametrization of a Pappus solid $P(C,\mathcal{L},\rho)$. Following is a direct application of Theorem \ref{manifold material vector}.

\begin{prop}
A parametrization $\vc{r}(t,x_1,\dots,x_{n-1})$ for a Pappus solid along a centroid curve $\vc{c}(t)$ is shown as :
\begin{equation*}
  \vc{r}(t,x_1,\dots,x_{n-1})=\vc{c}(t)+F_t\begin{bmatrix}
                                                   0 \\
                                                   x_1 \\
                                                   \vdots \\
                                                   x_{n-1} \\
                                                   0
                                                 \end{bmatrix}.
\end{equation*}
for $t\in I$ and $(x_2,\dots,x_{n})$ varies in some region $D_t$ in $\mathbb{R}^{n-1}$ depending on $t$ with the fact that
\[
m_{\text{cen}}(L_t)\vc{c}(t)=\int_{L_t}\rho\vc{r}\,\dd{V}\!\left(\mathbb{E}^{n-1}\right),
\]
for each $t\in I$.
\end{prop}

\begin{prop}\label{pro8}
The centered mass $m_{\text{cen}}(L_t)$ of the leaf $L_t$ is
\begin{align}
&\int_{D_t}\rho x_i\,\dd{x}_1\cdots\dd{x}_{n-1}=0, \quad\emph{\mbox{if }}~ i=1,\dots,n-1 \\
&\int_{D_t}\rho \,\dd{x}_1\cdots\dd{x}_{n-1}=m_{\text{cen}}(L_t).
\end{align}
\end{prop}
\begin{proof}
The mass center of the section $L_t$ is $\vc{c}(t)$. Hence the mass center vector is
\begin{align*}
m_{\text{cen}}(L_t)\vc{c}(t)&=\int_{L_t}\rho\vc{r}\,\dd{V}\!\left(\mathbb{E}^{n-1}\right)
\\
&=\int_{D_t}\rho\{\vc{c}(t)+x_1\vc{n}_1(t)+\dots+x_{n-1}\vc{n}_{n-1}(t)\}\, \dd{x}_1\cdots\dd{x}_{n-1}
\\
&=\left(\int_{D_t}\rho\,\dd{x}_1\cdots\dd{x}_{n-1}\right)\vc{c}(t)
+\sum_{i=1}^{n-1}\left(\int_{D_t}\rho x_i\,\dd{x}_1\cdots\dd{x}_{n-1}\right)\vc{n}_i(t).
\end{align*}
The conclusion comes from the linear independency.
\end{proof}

\subsubsection{Derivative of the basis matrix}
\begin{lemma}
  The matrix $F_t$ given in \emph{(\ref{eqn:euclidean matrix F_t})} satisfies the sytem of differential equations
  \begin{equation*}\label{}
    {F_t}'=F_t\Lambda_t
  \end{equation*}
  where the matrix $\Lambda_t$ is of the form
  \begin{equation*}\label{}
    \Lambda_t
    =\left(
     \begin{array}{c|c|c}
       0 & \begin{array}{ccc}
             \alpha_1 & \cdots & \alpha_{n-1}
           \end{array} & 0 \\
       \hline
       \begin{array}{c}
         -\alpha_1 \\
         \vdots \\
         -\alpha_{n-1}
       \end{array} & \lambda_t & \begin{array}{c}
                                   0 \\
                                   \vdots \\
                                   0
                                 \end{array} \\
       \hline
       0 & \begin{array}{ccc}
             0 & \cdots & 0
           \end{array} & 0
     \end{array}
     \right)
  \end{equation*}
  for some skew-symmetric matrix $\lambda_t$.
\end{lemma}
\begin{proof} Set $\vc{n}_0:=\vc{n}$, then $\langle \vc{n}_i ,\vc{n}_j\rangle=0$ for different indexes $i, j$ implies $\frac{\dd}{\dd{t}}\langle \vc{n}_i ,\vc{n}_j\rangle=\langle \vc{n}_i' ,\vc{n}_j\rangle+\langle \vc{n}_i ,\vc{n}_j'\rangle=0$. So skew-symmetric of $\Lambda_t$ is obtained. And the relation
  \[
  {F_t}'=[\vc{n}'(t),{\vc{n}_1}'(t),\cdots,{\vc{n}_{n-1}}'(t),\vcz]
  =[\vc{n}(t),{\vc{n}_1}(t),\cdots,{\vc{n}_{n-1}}(t),\mathbb{E}_{n+1}]\Lambda_t
  \]
  shows that the last column of $\Lambda_t$ is the zero vector.
\end{proof}
\subsubsection{Proof of the Pappus Theorem}
In order to calculate the volume of the Pappus solid $P(C,\mathcal{L},\rho)$, we need to find the volume element $\dd{V}^n$ with respect to the variables $t,x_1,\dots,x_{n-1}$:
\[
\dd{V}^n=\sqrt{\left|\det\left\{(D\vc{r})^TD\vc{r}\right\}\right|}\,\,\dd{t}\,\dd{x}_1\cdots\dd{x}_{n-1},
\]
where $D\vc{r}$ is the Jacobian matrix of $\vc{r}$.

Let us calculate the Jacobian matrix $D\vc{r}$.
\begin{align*}\label{}
  D\vc{r} & =(\vc{r}_t,\vc{r}_{x_1},\ldots,\vc{r}_{x_{n-1}})\\
  & =[\vc{c}',\vcz,\cdots,\vcz]
  +\left(
   \begin{array}{cc}
     {F_t}'\left[
           \begin{array}{c}
             0 \\
             x_1 \\
             \vdots \\
             x_{n-1} \\
             0
           \end{array}
           \right] & F_t\left(
                            \begin{array}{c}
                               \begin{array}{ccc}
                               0 & \cdots & 0
                            \end{array} \\
                            \hline
                            \begin{array}{c}
                            \\I_{n-1} \\{}
                            \end{array} \\
                            \hline
                            \begin{array}{ccc}
                            0 & \cdots & 0
                            \end{array}
   \end{array}
   \right)
   \end{array}
   \right)
  \\
  & =[\vc{t},\vcz,\cdots,\vcz]
  +F_t\left(
   \begin{array}{cc}
     \Lambda_t\left[
           \begin{array}{c}
             0 \\
             x_1 \\
             \vdots \\
             x_{n-1} \\
             0
           \end{array}
           \right] &        \left(
                            \begin{array}{c}
                               \begin{array}{ccc}
                               0 & \cdots & 0
                            \end{array} \\
                            \hline
                            \begin{array}{c}
                            \\I_{n-1} \\{}
                            \end{array} \\
                            \hline
                            \begin{array}{ccc}
                            0 & \cdots & 0
                            \end{array}
   \end{array}
   \right)
   \end{array}
   \right)
  \\
  & =\left(F_t
     \left[
           \begin{array}{c}
             \cos\theta_t \\
             \beta_1 \\
             \vdots \\
            \beta_{n-1} \\
             0
           \end{array}
           \right],\vcz,\ldots,\vcz\right)
  +F_t\left(
   \begin{array}{cc}
     \Lambda_t\left[
           \begin{array}{c}
             0 \\
             x_1 \\
             \vdots \\
             x_{n-1} \\
             0
           \end{array}
           \right] &        \left(
                            \begin{array}{c}
                               \begin{array}{ccc}
                               0 & \cdots & 0
                            \end{array} \\
                            \hline
                            \begin{array}{c}
                            \\I_{n-1} \\{}
                            \end{array} \\
                            \hline
                            \begin{array}{ccc}
                            0 & \cdots & 0
                            \end{array}
   \end{array}
   \right)
   \end{array}
   \right)
  \\
  & =F_t\left(
   \begin{array}{c|c}
     \cos\theta_t+\alpha_1x_1+\dots+\alpha_{n-1}x_{n-1} & \begin{array}{ccc}
                                         0 & \cdots & 0
                                       \end{array} \\
     \hline
                  \begin{pmatrix}
                  \beta_1 \\
                  \vdots \\
                  \beta_{n-1}
               \end{pmatrix}+ \lambda_t \begin{pmatrix}
                  x_1 \\
                  \vdots \\
                  x_{n-1}
               \end{pmatrix}& I_{n-1} \\
     \hline
     0 & \begin{array}{ccc}
           0 & \cdots & 0
         \end{array}
   \end{array}
   \right)=:F_t A
\end{align*}

%
Hence,
\begin{align*}
\dd{V}\left(\mathbb{E}^n\right)
&=\sqrt{\left|\det\left\{(D\vc{r})^TD\vc{r}\right\}\right|}\,\,\dd{x}_1\cdots\dd{x}_{n-1}\,\dd{t}\\
&=\sqrt{\left|\det(A^T A)\right|}\,\,\dd{x}_1\cdots\dd{x}_{n-1}\,\dd{t}\\
&=\left|\cos\theta_t+\alpha_1x_1+\dots+\alpha_{n-1}x_{n-1}\right|\,\dd{x}_1\cdots\dd{x}_{n-1}\dd{t}\\
&=\left(\cos\theta_t+\alpha_1x_1+\dots+\alpha_{n-1}x_{n-1}\right)\,\dd{x}_1\cdots\dd{x}_{n-1}\dd{t}.\\
\end{align*}
In conclusion, the above $\dd{V}\left(\mathbb{E}^n\right)$ formula and Proposition \ref{pro8} derives the Euclidean Pappus' Centroid theorem:
\begin{align*}\label{}
   &m_{\text{tot}}\Big(\mathcal{P},\rho\Big)
\\
=&\int_{\mathcal{P}}\rho\,\dd{V}\left(\mathbb{E}^n\right)
\\
   =&\int_{I}\int_{D_t}\rho(\cos\theta_t+\alpha_1x_1+\dots+\alpha_{n-1}x_{n-1})\, \dd{x}_1\cdots\dd{x}_{n-1}\dd{t}.
  \\
  =&\int_{I} \cos\theta_t\left(\int_{D_t}\rho \,\dd{x}_1\cdots\dd{x}_{n-1}\right)\dd{t}
    +\sum_{i=1}^{n-1}\int_{I}\alpha_i
    \left(\int_{D_t} \rho x_i\,\dd{x}_1\cdots\dd{x}_{n-1}\right)\dd{t}
  \\
   =&\int_{\vc{c}(t)}m_{\text{cen}}(L_t)\cos\theta_t\,\,\dd t.
\end{align*}

\subsection{Application of the Pappus' Theorem}
In this section, we will examine applications of the Pappus' theorem in non-Euclidean spaces. While the Pappus' theorem has been well studied in Euclidean spaces, in non-Euclidean spaces, there have been results only by Gray, Miquel \cite{gray}. Now, we will present several examples from the view-point of non-Euclidean mass center.

Here, we reconsider the Pappus' centroid theorem (\ref{pappus}) for Pappus' solid $\mathcal{P}$:
 \begin{align*}
    m_{\text{tot}}\Big(\mathcal{P}\big[\vc{c}(t),L_t,\theta_t)\big],\rho\Big)
    &=m_{\text{tot}}\Big(\vc{c}(t),\tilde{\rho}=m_{\text{cen}}(L_t,\rho) \cos\theta_t\Big).\\
 \end{align*}
If the density is trivial case, i.e., $\rho=1$, we call  $m_{\text{tot}}(\mathcal{P})$ as volume of $\mathcal{P}$ with notation $\text{Vol}(\mathcal{P})$.  The following four sections show the use of Pappus' theorem to determine the volume of Pappus solid with $\theta_t=0$. As far as we know, the following examples have not been previously documented, despite being very simple objects.

\subsubsection{Volume of a non-Euclidean solid torus}
If the radius of its circular cross section is $r$, and the radius of the circle traced by the mass center of the cross sections is $R$, then the volume $V$ of the Euclidean torus is well-known and that is $2\pi^2 R r^2$ by the Euclidean Pappus theorem.

Pappus' centroid theorem shows below volume formulas of non-Euclidean solid torus.

\begin{align*}
\text{Vol}(\text{spherical solid torus})
=&\int_{S^1(R)} ~m_{\text{cen}}(B^2(r)) ~  \dd t\\
=&\int_0^{2\pi \sin R} \pi \sin^2 r ~  \dd t\\
=&2\pi^2 \sin R\sin^2 r,\\
\end{align*}
\begin{align*}
\text{Vol}(\text{hyperbolic solid torus})
=&\int_{S^1(R)} ~m_{\text{cen}}(B^2(r)) ~  \dd t\\
=&\int_0^{2\pi \sinh R} \pi \sinh^2 r ~  \dd t\\
=&2\pi^2 \sinh R\sinh^2 r.\\
\end{align*}

\subsubsection{Volume of a non-Euclidean right circular cone}\label{rcc}
Already there is a well-known volume formula for a Euclidean right circular cone with radius $r$ and height $h$, which is $\frac{1}{3}(\pi r^2)h$. However, we could not find the volume formula for a non-Euclidean right circular cone in any other context.

For a spherical right circular cone  with radius $r$ and height $h$, the right cone has the following volume:
\begin{equation}\label{cone}
\pi \left( h-\frac{\sin h}{\sqrt{\tan^2 r +\sin^2 h}} \tan^{-1}\left( \frac{\sqrt{\tan^2 r +\sin^2 h}}{\cos h}\right) \right). \end{equation}
In order to find the formula, we adapt the Pappus' centroid theorem (\ref{pappus' centroid theorem}). First, we consider a spherical right triangle with standard notation three sides $a, b, c$ and three angles $A, B, C(=\frac{\pi}{2})$, then we get (see \cite{todhunter}, Art. 62)
\begin{equation}\label{right triangle}
    \cot A=\cot a \sin b.
\end{equation}
	\begin{figure}[h!]
		\begin{center}\label{fig:2266}
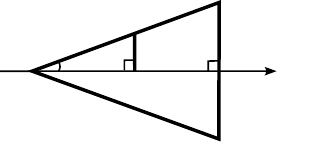

\caption{the triangular cross-section passing through the apex of a right circular cone}
  \end{center}
	\end{figure}
From the right triangle formed by a plane cut passing through the apex and center of the base circle of the cone, we obtain the above figure and the following relation from (\ref{right triangle}):
\begin{equation}\label{triangle}
    \cot\varphi=\cot r \sin h=\cot \alpha\sin t .
\end{equation}
So the spherical right circular cone with radius $r$ and height $h$ becomes a Pappus solid $\mathcal{P}=\mathcal{P}\big[\vc{c}(t),L_t,\theta_t\big]$, where $\vc{c}(t)=t ~(0\le t\le h)$ and $L_t=B^2(\alpha)$ and $\theta_t=0$. Note that from the table of $B^k(r)$ and the formula (\ref{triangle}), the centered mass of $L_t$ could be obtained by
\begin{align*} m_{\text{cen}}(L_t)&=\pi \sin^2 \alpha  \\
 &=\frac{\pi}{1+\cot^2 \alpha}\\
 &=\frac{\pi}{1+(\cot r \sin h)^2 \csc^2 t }.
\end{align*}
Hence Pappus' theorem derives that
\begin{align*}
\text{Vol}(\mathcal{P})
=&\int_0^h m_{\text{cen}}(L_t) ~  \dd t\\
=&\pi \int_0^h \frac{\dd t}{1+a^2 \csc^2 t }, \quad\quad a= \cot r \sin h   \\
=&\pi \left[  t-\frac{a}{\sqrt{1+a^2}}\tan^{-1}\left( \frac{\sqrt{1+a^2}}{a} \tan t \right)  \right] ^h_0   \\
=&\pi \left( h-\frac{\sin h}{\sqrt{\tan^2 r +\sin^2 h}} \tan^{-1}\left( \frac{\sqrt{\tan^2 r +\sin^2 h}}{\cos h}\right) \right).
\end{align*}

Consider a hyperbolic right circular cone  with radius $r$ and height $h$, the situation is very similar to the spherical case. We have to change the formulas (\ref{right triangle}) and (\ref{triangle}) to  $$\cot A=\coth a \sinh b,\quad  \cot\varphi=\coth r \sinh h=\coth \alpha\sinh t.$$
Then $L_t$ is similarly obtained by
\begin{align*} m_{\text{cen}}(L_t)&=\pi \sinh^2 \alpha  \\
 &=\frac{\pi}{\coth^2 \alpha -1}\\
 &=\frac{\pi}{(\coth r \sinh h)^2 \,\text{csch}^2 t -1}.
\end{align*}
Hence Pappus' theorem also derives that
\begin{align*}
\text{Vol}(\mathcal{P})
=&\pi \int_0^h \frac{\dd t}{a^2\, \text{csch}^2 t -1}, \quad\quad a= \coth r \sinh h   \\
=&\pi \left[  -t+\frac{a}{\sqrt{1+a^2}}\tanh^{-1}\left( \frac{\sqrt{1+a^2}}{a} \tanh t \right)  \right] ^h_0   \\
=&\pi \left( -h+\frac{\sinh h}{\sqrt{\tanh^2 r +\sinh^2 h}} \tanh^{-1}\left( \frac{\sqrt{\tanh^2 r +\sinh^2 h}}{\cosh h}\right) \right).
\end{align*}

\subsubsection{Volume of an $n$-dimensional non-Euclidean right cone with a $(n-1)$-dimensional ball base}
There is a well-known volume formula for a Euclidean right cone with base $B^{n-1}(r)$ and height $h$, which is $\frac{1}{n} \text{Vol}(B^{n-1}(r))h=\frac{1}{n}(\frac{\pi^{\frac{n-1}{2}}}{\Gamma(\frac{n+1}{2})}  r^{n-1})h$.

For a spherical right cone with radius $r$ disk and height $h$, the right cone has the following volume:
\begin{equation}\label{general cone}
\frac{\pi^{\frac{n-1}{2}}}{\Gamma(\frac{n+1}{2})} \int_0^h \frac{\dd t}{(a^2 \csc^2 t +1)^{\frac{n-1}{2}}}, \quad\quad a= \cot r \sin h .
\end{equation}
In order to find the formula, we need a similar process showed in the example  \ref{rcc} with  $L_t=B^{n-1}(\alpha)$ instead of $L_t=B^{2}(\alpha)$.

Hence Pappus theorem derives that
\begin{align*}
&\int_0^h m_{\text{cen}}(L_t) ~  \dd t\\
=&\int_0^h \frac{\pi^{\frac{n-1}{2}}}{\Gamma(\frac{n+1}{2})} \sin^{n-1}  \alpha ~  \dd t\\
=&\frac{\pi^{\frac{n-1}{2}}}{\Gamma(\frac{n+1}{2})} \int_0^h \frac{\dd t}{(a^2 \csc^2 t+1)^{\frac{n-1}{2}}}, \quad\quad a= \cot r \sin h   \\
\end{align*}
For a hyperbolic right cone with radius $r$ disk and height $h$, the right cone has the following volume similar to (\ref{general cone}):
\begin{equation*}
\frac{\pi^{\frac{n-1}{2}}}{\Gamma(\frac{n+1}{2})} \int_0^h \frac{\dd t}{(a^2\, \text{csch}^2 t -1)^{\frac{n-1}{2}}}, \quad\quad a= \coth r \sinh h .
\end{equation*}

\subsubsection{Volume of a non-Euclidean right cone with a regular $n$-gonal base}
For a spherical right cone  with a regular $n$-gonal $P_n^{\mathbb{S}}(a)$ base and  height $h$, the cone has the following volume:
\begin{equation}\label{cone2}
\int^h_0 \frac{n}{\sqrt{1+\cot^2 b\sin^2 h \csc^2 t}}\tan^{-1}\left( \frac{\tan\frac{\pi}{n}}{\sqrt{1+\cot^2 b\sin^2 h \csc^2 t}}\right)~\dd t,  \end{equation}
where $\cot^2 b= \tan^2 \frac{\pi}{n} \cot^2 \frac{a}{2} -1$.
\begin{figure}[h!]
		\begin{center}\label{fig:10}
   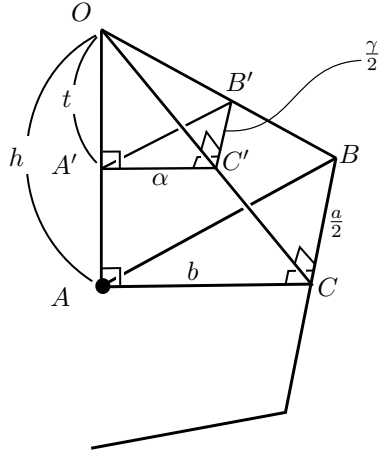

   \caption{a spherical or hyperbolic right cone with a regular $n$-gonal base}
		\end{center}
	\end{figure}

First, we consider a spherical right triangle $\triangle OAC$, where the point $O$ is the apex of the cone, the point $A$ is the center of the base polygon, and the point $C$ is a midpoint of one side edge.

Here the spherical right cone over the regular polygon $P_n^{\mathbb{S}}(a)$ with height $h$ becomes a Pappus solid $\mathcal{P}=\mathcal{P}\big[\vc{c}(t),L_t,\theta_t\big]$, where $\vc{c}(t)=t ~(0\le t\le h)$ and $L_t=P_n(\gamma)$ and $\theta_t=0$.  Therefore, according to Pappus' centroid theorem (\ref{pappus' centroid theorem}), the volume of the cone can be expressed as follows:
\begin{equation*}
\int_0^h m_{\text{cen}}(P_n(\gamma)) ~  \dd t
=\int_0^h \frac{n\gamma}{2}\tan  \frac{\gamma}{2} \cot \frac{\pi}{n} \dd t. \end{equation*}
Let us denote the right triangle $\triangle A'B'C'$ of the section $L_t=P_n(\gamma)$  corresponding to the right triangle $\triangle ABC$ of the base $P_n^{\mathbb{S}}(a)$, where $\overline{B'C'}=\frac{\gamma}{2}$ and define $\overline{A'C'}=\alpha$. Then $\alpha$ and the $\gamma$ are obtained by
\begin{align*}
    \sin^2 \alpha =& \frac{1}{1+\cot^2 \alpha}= \frac{1}{1+\cot^2 b \sin^2 h \csc^2 t},\\
    \tan\frac{\gamma}{2}=& \tan \frac{\pi}{n} \sin \alpha.
\end{align*}
Finally, we get the volume of the cone over the regular polygon $P_n^{\mathbb{S}}(a)$ with height $h$ as
\begin{align*}
&\int_0^h \frac{n\gamma}{2}\tan  \frac{\gamma}{2} \cot \frac{\pi}{n} \dd t\\
=&\int^h_0 \frac{n}{\sqrt{1+\cot^2 b\sin^2 h \csc^2 t}}\tan^{-1}\left( \frac{\tan\frac{\pi}{n}}{\sqrt{1+\cot^2 b\sin^2 h \csc^2 t}}\right)~\dd t.
\end{align*}

For a hyperbolic right cone  with a regular $n$-gonal $P_n^{\mathbb{H}}(a)$ base and  height $h$, the cone has the following volume:
\begin{align*}
&\int_0^h \frac{n\gamma}{2}\tanh  \frac{\gamma}{2} \cot \frac{\pi}{n} \dd t\\
=&\int^h_0 \frac{n}{\sqrt{-1+\coth^2 b\sinh^2 h \text{ csch}^2 t}}\tanh^{-1}\left( \frac{\tan\frac{\pi}{n}}{\sqrt{-1+\coth^2 b \sinh^2 h \text{ csch}^2 t}}\right)~\dd t,
\end{align*}
where $\coth^2 b= \tan^2 \frac{\pi}{n} \coth^2 \frac{a}{2} +1$. This case is left as an exercise for the reader.
\vskip 1pc

The non-Euclidean volume can be computed not only for the objects introduced above but also for various symmetric shapes. However, deriving the volume formula without relying on Pappus’ centroid theorem reveals the inherent difficulty of such calculations, even for simple shapes like a right circular cone or a solid torus.



\vskip 1pc
\noindent Department of Mathematics, University of Seoul, Seoul 02504, Korea\\
{\it Email address}: \texttt{yhcho@uos.ac.kr}\\

\noindent Faculty of Liberal Education, Seoul National University,
 Seoul 08826, Korea \\
{\it  Email address}: \texttt{hgchoi66@snu.ac.kr}
\vfill\eject

\vfill\eject

\end{document}

%% file: spaces_basic1.tex
\begin{center}
    \begin{tikzpicture}[scale=1.3]
    \draw (0, -0.5) -- (0, 2) node[anchor=south]{$\mathbb{R}$};
    \draw (-2,0) -- (2,0) node[anchor=west]{$\mathbb{R}^n$};

    \draw (0,1) node[anchor=south east]{$1$};

    \draw[very thick] (-2,1) node[anchor=south]{$\mathbb{E}^n$}  -- (2,1);
    \draw[blue, thick, ->] (0,0) -- (2,2);

    \draw (0,0) to [bend left=30] (1,1);
    \draw (0.5,0.8) node {\small $1$};

    \draw (0,0) to [bend left=30] (2,2);
    \draw (1.1,1.8) node {\small $m_{\vc a}$};

    \filldraw[blue] (1,1) circle (1pt);
    \draw (2.5, 2.2) node {$\vc a = m_{\vc a}\vc r$};
    \draw (0.9,0.7) node[anchor=west] {$[\vc a] = \vc r$};
  \end{tikzpicture}
\end{center}

%% file: spaces_basic2.tex
\begin{center}
  \begin{tikzpicture}[scale=1.3]
    \draw (0, -1.3) -- (0, 2) node[anchor=south]{$\mathbb{R}$};
    \draw (-1.5,0) -- (2,0) node[anchor=west]{$\mathbb{R}^n$};

    \draw (0,1) node[anchor=south east]{$1$};

    \draw[very thick, domain=0:360, smooth, variable=\t, samples=100] 
      plot ({cos(\t)}, {sin(\t)});

    \draw (-0.7, 0.7) node[anchor=east]{$\mathbb{S}^n$};

    \draw[blue, thick, ->, trig format=rad] (0,0) -- ({2*cos(pi/7)}, {2*sin(pi/7)});

    \draw (0,0) to [bend left=30] ({cos(deg(pi/7))}, {sin(deg(pi/7))});
    \draw (0.6,{sin(deg(pi/7))}) node {\small $1$};

    \draw (0,0) to [bend left=30] ({2*cos(deg(pi/7))}, {2*sin(deg(pi/7))});
    \draw (1.2,{2*sin(deg(pi/7))+0.1}) node {\small $m_{\vc a}$};

    \filldraw[blue] ({cos(deg(pi/7))}, {sin(deg(pi/7))}) circle (1pt);
    \draw (2, 0.9) node { $\vc a$};
    \draw (1.2, 0.4) node { $[\vc a]$};
  \end{tikzpicture}
  \begin{tikzpicture}[scale=1.3]
    \draw (0, -1.3) -- (0, 2) node[anchor=south]{$\mathbb{R}$};
    \draw (-2,0) -- (2,0) node[anchor=west]{$\mathbb{R}^n$};

    \draw[very thick, domain=-1.5:1.5, smooth, variable=\t, samples=100] 
      plot ({sinh(\t)}, {cosh(\t)}) node[anchor=west]{$\mathbb{H}^n$};

    \draw[blue, thick, ->, trig format=rad] (0,0) -- ({2*sinh(0.8)}, {2*cosh(0.8)});

    \draw (0,0) to [bend left=15] ({sinh(0.8)}, {cosh(0.8)});
    \draw (0.25,{0.5*cosh(0.8)+0.1}) node {\small $1$};

    \draw (0,0) to [bend left=30] ({2*sinh(0.8)}, {2*cosh(0.8)});
    \draw (0.4,{cosh(0.8)+0.5}) node {\small $m_{\vc a}$};

    \draw ({2*sinh(0.75)}, {2*cosh(0.8)}) node[anchor=south west] { $\vc a$};
    \filldraw[blue] ({sinh(0.8)}, {cosh(0.8)}) circle (1pt);
    \draw (1.1, 1.1) node { $[\vc a]$};
  \end{tikzpicture}
\end{center}

%% file: triangle3.tex
\begin{tikzpicture}
  \coordinate (A) at (0, 0);
  \coordinate (B) at (10, 0);
  \coordinate (C) at (4, 5);
  \coordinate (D) at (3.4, 0);
  \coordinate (E) at (2, {1.25*2});
  \coordinate (G) at (7, 0);

  \draw[thick] (A) node[anchor=north] {$A = [(p + 1) \vc{u}]$}
    -- (B) node[anchor=north] {$B = [q \vc{v}]$}
    -- (C) node[anchor=south] {$C = [\vc{w}]$} 
    -- cycle;

  \draw[thick, name path=C--D] (C) 
    -- (D) node[anchor=north] {$D = [(p + 1) \vc{u} \oplus q \vc{v}]$};

  \draw[thick, name path=E--G] (E) node[anchor=south east, yshift=-4] 
    {$\begin{aligned}
        E = [\vc {u} \oplus \vc {w}]\\ = [\vc {u} + \vc {w}]
    \end{aligned}$}
    -- (G) node[anchor=north] {$\begin{aligned}
      G = [p \vc{u} \oplus q \vc{v}]\\ = [p \vc{u} + q \vc{v}]
    \end{aligned}$};

  \filldraw[black] (A) circle (1pt)
    (B) circle (1pt)
    (C) circle (1pt)
    (D) circle (1pt)
    (E) circle (1pt)
    (G) circle (1pt);

  \coordinate (P) at (intersection of C--D and E--G);
  \filldraw[black] (P) circle (2pt)
    node[anchor=west, xshift=8] {$P = [(p + 1) \vc{u} \oplus q \vc{v} \oplus \vc{w}]$};

  \path[name path=A--P] (A) -- (P);
  \path[name path=B--C] (B) -- (C);
  \coordinate (F) at (intersection of A--P and B--C);

  \filldraw[black] (F) circle (1pt) 
    node[anchor=west, yshift=3] {$F = [q \vc{v} \oplus \vc{w} 
      = [q \vc{v} + \vc{w}]$};
  \draw[thick] (A) -- (F);
\end{tikzpicture}

%% file: circle1.tex
\begin{tikzpicture}[scale=1.2]
  \def \r {2} 

  \coordinate (u1) at (\r, 0);
  \coordinate (u2) at (-\r, 0);
  \coordinate (v1) at ({\r * cos(45)}, {\r * sin(45)});
  \coordinate (v2) at ({\r * cos(135)}, {\r * sin(135)});
  \coordinate (w) at (0, \r);

  \draw[vechead, thick] (-4, 0) -- (4, 0);
  \draw[vechead, thick] (0, -3) -- (0, 3);
  \draw[thick] circle (\r);

  \filldraw[black] (u1) circle (1pt)
    node[anchor=north west] {$\vc{u_1}$};
  \filldraw[black] (u2) circle (1pt)
    node[anchor=north east] {$\vc{u_2}$};
  \filldraw[black] (v1) circle (1pt)
    node[anchor=south west] {$\vc{v_1}$};
  \filldraw[black] (v2) circle (1pt)
    node[anchor=south east] {$\vc{v_2}$};
  \filldraw[black] (w) circle (1pt)
    node[anchor=south west] {$\vc{w}$};

  \draw[thick] (0, 0) -- (v1);
  \draw[thick] (0, 0) -- (v2);

  \draw (0.4, 0) arc (0:45:0.4) node[midway, anchor=west, yshift=4] {$\frac{d}{2}$};
  \draw ({0.4 * cos(45)}, {0.4 * sin(45)}) 
    .. controls ({0.5 * cos(55)}, {0.5 * sin(55)}) and ({0.5 * cos(80)}, {0.5 * sin(80)}) 
    .. (0, 0.4) node[midway, anchor=south, xshift=8, yshift=10] 
      {$\frac{\pi}{2} - \frac{d}{2}$};
  \draw (0, 0.4) 
    .. controls ({0.5 * cos(100)}, {0.5 * sin(100)}) 
      and ({0.5 * cos(125)}, {0.5 * sin(125)}) 
    .. ({0.4 * cos(135)}, {0.4 * sin(135)}) node[midway, anchor=south, xshift=-8, yshift=10]
      {$\frac{\pi}{2} - \frac{d}{2}$};
  \draw (-0.4, 0) arc (180:135:0.4) node[midway, anchor=east, yshift=4] {$\frac{d}{2}$};
\end{tikzpicture}

%% file: line1.tex
\begin{tikzpicture}
    \coordinate (A) at (1, 0);
    \coordinate (B) at (7, 0);
    \coordinate (P) at (3.5, 0);
    
    \draw[very thick] (0, 0) -- (8, 0);
    \filldraw[black] (A) circle (1.5pt) node[anchor=south] {$[\vc {a}]$};
    \filldraw[black] (B) circle (1.5pt) node[anchor=south] {$[\vc {b}]$};
    \filldraw[black] (P) circle (1.5pt) node[anchor=north] {$[\vc {a} \oplus \vc {b}]$};
    
    \draw (A) to [out=25, in=155, distance=1cm] node[midway, above] {$d_1$} (P);
    \draw (P) to [out=25, in=155, distance=1cm] node[midway, above] {$d_2$} (B);
    \draw (A) to [out=155, in=25, distance=-2cm] node[midway, below] {$d$} (B);
\end{tikzpicture}

%% file: polygonTriangle.pdf_tex
\begingroup%
  \makeatletter%
  \providecommand\color[2][]{%
    \errmessage{(Inkscape) Color is used for the text in Inkscape, but the package 'color.sty' is not loaded}%
    \renewcommand\color[2][]{}%
  }%
  \providecommand\transparent[1]{%
    \errmessage{(Inkscape) Transparency is used (non-zero) for the text in Inkscape, but the package 'transparent.sty' is not loaded}%
    \renewcommand\transparent[1]{}%
  }%
  \providecommand\rotatebox[2]{#2}%
  \newcommand*\fsize{\dimexpr\f@size pt\relax}%
  \newcommand*\lineheight[1]{\fontsize{\fsize}{#1\fsize}\selectfont}%
  \ifx\svgwidth\undefined%
    \setlength{\unitlength}{296.78539103bp}%
    \ifx\svgscale\undefined%
      \relax%
    \else%
      \setlength{\unitlength}{\unitlength * \real{\svgscale}}%
    \fi%
  \else%
    \setlength{\unitlength}{\svgwidth}%
  \fi%
  \global\let\svgwidth\undefined%
  \global\let\svgscale\undefined%
  \makeatother%
  \begin{picture}(1,0.43371339)%
    \lineheight{1}%
    \setlength\tabcolsep{0pt}%
    \put(0.93480698,0.33033917){\color[rgb]{0,0,0}\makebox(0,0)[lt]{\lineheight{0}\smash{\begin{tabular}[t]{l}$B$\end{tabular}}}}%
    \put(0.29017288,0.29258747){\color[rgb]{0,0,0}\makebox(0,0)[lt]{\lineheight{0}\smash{\begin{tabular}[t]{l}$B$\end{tabular}}}}%
    \put(0.28911905,0.24898599){\color[rgb]{0,0,0}\makebox(0,0)[lt]{\lineheight{0}\smash{\begin{tabular}[t]{l}$\frac{a}{2}$\end{tabular}}}}%
    \put(0.19412055,0.22701694){\color[rgb]{0,0,0}\makebox(0,0)[lt]{\lineheight{0}\smash{\begin{tabular}[t]{l}$\frac{\pi}{n}$\end{tabular}}}}%
    \put(0.02705707,0.4010286){\color[rgb]{0,0,0}\makebox(0,0)[lt]{\lineheight{0}\smash{\begin{tabular}[t]{l}$a$\end{tabular}}}}%
    \put(0.22295999,0.37672113){\color[rgb]{0,0,0}\makebox(0,0)[lt]{\lineheight{0}\smash{\begin{tabular}[t]{l}$a$\end{tabular}}}}%
    \put(0.18638867,0.26532044){\color[rgb]{0,0,0}\makebox(0,0)[lt]{\lineheight{0}\smash{\begin{tabular}[t]{l}$a$\end{tabular}}}}%
    \put(0.74428819,0.08466615){\color[rgb]{0,0,0}\makebox(0,0)[lt]{\lineheight{0}\smash{\begin{tabular}[t]{l}$b$\end{tabular}}}}%
    \put(0.18212451,0.18247356){\color[rgb]{0,0,0}\makebox(0,0)[lt]{\lineheight{0}\smash{\begin{tabular}[t]{l}$b$\end{tabular}}}}%
    \put(0.52213619,0.10991824){\color[rgb]{0,0,0}\makebox(0,0)[lt]{\lineheight{0}\smash{\begin{tabular}[t]{l}$A$\end{tabular}}}}%
    \put(0.05582362,0.19580115){\color[rgb]{0,0,0}\makebox(0,0)[lt]{\lineheight{0}\smash{\begin{tabular}[t]{l}$A$\end{tabular}}}}%
    \put(0,0){\includegraphics[width=\unitlength,page=1]{polygonTriangle.pdf}}%
    \put(0.93914349,0.2235272){\color[rgb]{0,0,0}\makebox(0,0)[lt]{\lineheight{0}\smash{\begin{tabular}[t]{l}$D$\end{tabular}}}}%
    \put(0.80442145,0.17031607){\color[rgb]{0,0,0}\makebox(0,0)[lt]{\lineheight{0}\smash{\begin{tabular}[t]{l}$r$\end{tabular}}}}%
    \put(0.93893886,0.10059649){\color[rgb]{0,0,0}\makebox(0,0)[lt]{\lineheight{0}\smash{\begin{tabular}[t]{l}$C$\end{tabular}}}}%
    \put(0.29357358,0.19563807){\color[rgb]{0,0,0}\makebox(0,0)[lt]{\lineheight{0}\smash{\begin{tabular}[t]{l}$C$\end{tabular}}}}%
    \put(0,0){\includegraphics[width=\unitlength,page=2]{polygonTriangle.pdf}}%
    \put(0.66632763,0.12605394){\color[rgb]{0,0,0}\makebox(0,0)[lt]{\lineheight{0}\smash{\begin{tabular}[t]{l}$\varphi$\end{tabular}}}}%
    \put(0,0){\includegraphics[width=\unitlength,page=3]{polygonTriangle.pdf}}%
  \end{picture}%
\endgroup%

%% file: triangle4.pdf_tex
\begingroup%
  \makeatletter%
  \providecommand\color[2][]{%
    \errmessage{(Inkscape) Color is used for the text in Inkscape, but the package 'color.sty' is not loaded}%
    \renewcommand\color[2][]{}%
  }%
  \providecommand\transparent[1]{%
    \errmessage{(Inkscape) Transparency is used (non-zero) for the text in Inkscape, but the package 'transparent.sty' is not loaded}%
    \renewcommand\transparent[1]{}%
  }%
  \providecommand\rotatebox[2]{#2}%
  \newcommand*\fsize{\dimexpr\f@size pt\relax}%
  \newcommand*\lineheight[1]{\fontsize{\fsize}{#1\fsize}\selectfont}%
  \ifx\svgwidth\undefined%
    \setlength{\unitlength}{157.84080855bp}%
    \ifx\svgscale\undefined%
      \relax%
    \else%
      \setlength{\unitlength}{\unitlength * \real{\svgscale}}%
    \fi%
  \else%
    \setlength{\unitlength}{\svgwidth}%
  \fi%
  \global\let\svgwidth\undefined%
  \global\let\svgscale\undefined%
  \makeatother%
  \begin{picture}(1,0.43013716)%
    \lineheight{1}%
    \setlength\tabcolsep{0pt}%
    \put(0.3915075,0.15673196){\color[rgb]{0,0,0}\makebox(0,0)[lt]{\lineheight{0}\smash{\begin{tabular}[t]{l}$t$\end{tabular}}}}%
    \put(0.06178788,0.15383783){\color[rgb]{0,0,0}\makebox(0,0)[lt]{\lineheight{0}\smash{\begin{tabular}[t]{l}$0$\end{tabular}}}}%
    \put(0.6819695,0.15889363){\color[rgb]{0,0,0}\makebox(0,0)[lt]{\lineheight{0}\smash{\begin{tabular}[t]{l}$h$\end{tabular}}}}%
    \put(0,0){\includegraphics[width=\unitlength,page=1]{triangle4.pdf}}%
    \put(0.42392852,0.26458099){\color[rgb]{0,0,0}\makebox(0,0)[lt]{\lineheight{0}\smash{\begin{tabular}[t]{l}$\alpha$\end{tabular}}}}%
    \put(0.68526779,0.30734558){\color[rgb]{0,0,0}\makebox(0,0)[lt]{\lineheight{0}\smash{\begin{tabular}[t]{l}$\gamma$\end{tabular}}}}%
    \put(0.23958808,0.22876211){\color[rgb]{0,0,0}\makebox(0,0)[lt]{\lineheight{0}\smash{\begin{tabular}[t]{l}$\varphi$\end{tabular}}}}%
  \end{picture}%
\endgroup%

%% file: pyramid.pdf_tex
\begingroup%
  \makeatletter%
  \providecommand\color[2][]{%
    \errmessage{(Inkscape) Color is used for the text in Inkscape, but the package 'color.sty' is not loaded}%
    \renewcommand\color[2][]{}%
  }%
  \providecommand\transparent[1]{%
    \errmessage{(Inkscape) Transparency is used (non-zero) for the text in Inkscape, but the package 'transparent.sty' is not loaded}%
    \renewcommand\transparent[1]{}%
  }%
  \providecommand\rotatebox[2]{#2}%
  \newcommand*\fsize{\dimexpr\f@size pt\relax}%
  \newcommand*\lineheight[1]{\fontsize{\fsize}{#1\fsize}\selectfont}%
  \ifx\svgwidth\undefined%
    \setlength{\unitlength}{215.16588032bp}%
    \ifx\svgscale\undefined%
      \relax%
    \else%
      \setlength{\unitlength}{\unitlength * \real{\svgscale}}%
    \fi%
  \else%
    \setlength{\unitlength}{\svgwidth}%
  \fi%
  \global\let\svgwidth\undefined%
  \global\let\svgscale\undefined%
  \makeatother%
  \begin{picture}(1,0.78482169)%
    \lineheight{1}%
    \setlength\tabcolsep{0pt}%
    \put(0.41316584,0.62802021){\color[rgb]{0,0,0}\makebox(0,0)[lt]{\lineheight{0}\smash{\begin{tabular}[t]{l}$B'$\end{tabular}}}}%
    \put(0.61153381,0.50249974){\color[rgb]{0,0,0}\makebox(0,0)[lt]{\lineheight{0}\smash{\begin{tabular}[t]{l}$B$\end{tabular}}}}%
    \put(0.59215383,0.38808173){\color[rgb]{0,0,0}\makebox(0,0)[lt]{\lineheight{0}\smash{\begin{tabular}[t]{l}$\frac{a}{2}$\end{tabular}}}}%
    \put(0.65382773,0.68219657){\color[rgb]{0,0,0}\makebox(0,0)[lt]{\lineheight{0}\smash{\begin{tabular}[t]{l}$\frac{\gamma}{2}$\end{tabular}}}}%
    \put(0.28400304,0.46462117){\color[rgb]{0,0,0}\makebox(0,0)[lt]{\lineheight{0}\smash{\begin{tabular}[t]{l}$\alpha$\end{tabular}}}}%
    \put(0.3459777,0.29708588){\color[rgb]{0,0,0}\makebox(0,0)[lt]{\lineheight{0}\smash{\begin{tabular}[t]{l}$b$\end{tabular}}}}%
    \put(0.10787138,0.48333189){\color[rgb]{0,0,0}\makebox(0,0)[lt]{\lineheight{0}\smash{\begin{tabular}[t]{l}$A'$\end{tabular}}}}%
    \put(0.10703263,0.25197982){\color[rgb]{0,0,0}\makebox(0,0)[lt]{\lineheight{0}\smash{\begin{tabular}[t]{l}$A$\end{tabular}}}}%
    \put(0,0){\includegraphics[width=\unitlength,page=1]{pyramid.pdf}}%
    \put(0.14278245,0.75207624){\color[rgb]{0,0,0}\makebox(0,0)[lt]{\lineheight{0}\smash{\begin{tabular}[t]{l}$O$\end{tabular}}}}%
    \put(0.40741211,0.49039224){\color[rgb]{0,0,0}\makebox(0,0)[lt]{\lineheight{0}\smash{\begin{tabular}[t]{l}$C'$\end{tabular}}}}%
    \put(0.57464526,0.26968128){\color[rgb]{0,0,0}\makebox(0,0)[lt]{\lineheight{0}\smash{\begin{tabular}[t]{l}$C$\end{tabular}}}}%
    \put(0,0){\includegraphics[width=\unitlength,page=2]{pyramid.pdf}}%
    \put(0.03381711,0.49654916){\color[rgb]{0,0,0}\makebox(0,0)[lt]{\lineheight{0}\smash{\begin{tabular}[t]{l}$h$\end{tabular}}}}%
    \put(0.12725641,0.6002525){\color[rgb]{0,0,0}\makebox(0,0)[lt]{\lineheight{0}\smash{\begin{tabular}[t]{l}$t$\end{tabular}}}}%
    \put(0,0){\includegraphics[width=\unitlength,page=3]{pyramid.pdf}}%
  \end{picture}%
\endgroup%